\documentclass[journal,twoside,web]{IEEEtran}
\usepackage{cite}
\usepackage{amsmath,amssymb,amsfonts}
\usepackage{amsthm}
\usepackage{textcomp}
\usepackage{algorithm}
\usepackage{algpseudocode}% http://ctan.org/pkg/algorithmicx
\usepackage{graphicx}
\usepackage{subcaption}
\usepackage{enumerate}
\usepackage{multirow}
\usepackage{hyperref}
\usepackage{epstopdf}
\allowdisplaybreaks

\algdef{SE}{Begin}{End}{\textbf{begin}}{\textbf{end}}

\newtheorem{theorem}{Theorem}[section]
\newtheorem{corollary}{Corollary}[theorem]
\newtheorem{lemma}[theorem]{Lemma}
\newtheorem{assumption}[theorem]{Assumption}

\def\BibTeX{{\rm B\kern-.05em{\sc i\kern-.025em b}\kern-.08em
    T\kern-.1667em\lower.7ex\hbox{E}\kern-.125emX}}
\begin{document}
\title{Dynamic Modeling and Parameter Estimation for\\Origami Structure Reconfiguration Process}
%\author{Changhuang Wan, \IEEEmembership{Fellow, IEEE}, Sixiong You, Gangshan JIng, and Ran Dai, \IEEEmembership{MemberFellow, IEEE}
\author{Yuto Tanaka, Ran Dai, and Mehran Mesbahi
	\thanks{Yuto Tanaka and Ran Dai are with the School of Aeronautics and Astronautics, Purdue University, West Lafayette, IN, 47907.
		Emails: {\tt\small tanaka9@purdue.edu, randai@purdue.edu}}
        \thanks{Mehran Mesbahi is with William E. Boeing Department of Aeronautics
    	and Astronautics, University of Washington, Seattle, WA, 98195. Email: {\tt \small mesbahi@uw.edu}}
        }
%\thanks{S. B. Author, Jr., was with Rice University, Houston, TX 77005 USA. He is 
%now with the Department of Physics, Colorado State University, Fort Collins, 
%CO 80523 USA (e-mail: author@lamar.colostate.edu).}
%\thanks{T. C. Author is with 
%the Electrical Engineering Department, University of Colorado, Boulder, CO 
%80309 USA, on leave from the National Research Institute for Metals, 
%Tsukuba, Japan (e-mail: author@nrim.go.jp).}}

\maketitle

\begin{abstract}
The reconfiguration of origami during the folding and unfolding process is governed through a sequence of panel deformations and hinge orientations. {To} develop an effective model for representing the reconfiguration process, this paper introduces planar straight-line graphs and a novel consensus protocol for reaching the target origami configuration. The convergence and stability properties of the proposed consensus protocol are subsequently analyzed. Furthermore, {to account for aggregate material and structural effects in the proposed consensus-based reconfiguration model, effective parameters embedded in the consensus protocol are identified from trajectory data using a fitting algorithm.}
Lastly, the effectiveness of the proposed modeling approach is shown using
simulations of the two-panel structure and the Kresling origami pattern reconfiguration process.
\end{abstract}
%
%\begin{IEEEkeywords}-Origami dynamics and reconfiguration; consensus protocol; parameter estimation; Kresling origami pattern
%\end{IEEEkeywords}

\section{Introduction}\label{s:introduction}
Origami, the traditional art of paper-folding, has recently become a multidisciplinary
framework for designing novel mechanisms and engineering structures.
Origami structural principles are used for example in robotics, biomedical devices, and aerospace engineering. From robotic systems \cite{rus2018spotlight} and foldable medical tools \cite{johnson2017fabricating} to compact satellite components \cite{hwang2020origami}, origami-based designs stand out for their portability, adaptability, and efficiency \cite{rus2018design,fei2013origami}. As the range of these applications expand, accurately modeling and predicting how origami structures behave under dynamic conditions has become essential.

Modeling approaches for origami systems generally fall into two categories: mechanics-based models,  built on structural and material principles, and data-driven models that rely on observed system behavior. Mechanics-based frameworks, such as bar-and-hinge models, are particularly effective in capturing the kinematic constraints and physical motions associated with folding structures \cite{liu2017bar, Filipov2017BarHinge, Woodruff2020BarHinge}. However, these models typically require detailed prior knowledge of material properties, such as Young's modulus, to accurately represent the system's behavior during the reconfiguration process. In addition, while bar-and-hinge models effectively characterize geometric and mechanical transformations, they often lack a comprehensive analysis of the system’s dynamic properties, indispensable for investigating their time-evolution.

Data-driven modeling approaches--including those based on the finite element method, machine learning and numerical techniques--have shown strong predictive capabilities for complex systems \cite{Hu2021FEM, Mosh2022Fatigue, Zhang2023ML, Fonseca2022ML2}. However, these methods typically rely heavily on large volumes of high-quality data to capture the system dynamics accurately. As a result, the generalizability of the aforementioned methods is often limited to the range and conditions captured by the training data. Moreover, generating and preprocessing such data, particularly through high-fidelity simulations can be computationally expensive and time-consuming. 
By combining the motivation of exploring origami dynamic models with insights from data-driven methods, hybrid modeling approaches hold promise for creating more accurate and versatile representations of origami structures.

It is natural to represent origami structures using graph theory, where creases correspond to edges and crease intersections map to graph vertices. This connection has inspired a broad body of work at the intersection of origami and graph theory. On one hand, unit origami has served as an intuitive tool to teach fundamental graph theoretic concepts~\cite{hull1995unit}; on the other hand, graph-theoretic methods have been employed to support origami design, providing mathematical insights into geometric and topological properties of foldable structures \cite{dureisseix2015example,kanade1980theory,turner2016review,xia2024novel,yamaguchi2022graph}.
This paper presents a systematic approach to {model} the folding and unfolding processes of origami structures using dynamic networks and a new twist on the consensus protocol. %The network-based approach defines the origami system's elements as vertices and edges in a planar straight-line graph, where the relative motion of vertices is governed by this new consensus protocol. 

In an origami system, reconfiguration is achieved through panel deformations, namely expansion and contraction within individual two-dimensional (2D) planes and rotations along shared edges. However, the traditional consensus protocol adopted for three-dimensional (3D) space steers the coordinates of the vertices independently along each axis; {specifically}, the protocol does not restrict origami movement due to expansion/contraction within individual 2D panel surfaces \cite{mesbahi2010consensus,ren2008distributed,lou2014approximate}. Furthermore, the traditional consensus protocol cannot be applied directly to coordinate vertices for panel rotational movements around the hinges, {as illustrated by comparative examples in the following context}.

This work proposes a frame-projected consensus protocol (FPCP) to capture the complex and localized interactions between vertices during the origami reconfiguration process. This method projects each vertex into the frame of its corresponding panel, hence the name ``frame projection'' consensus. FPCP is employed in two distinct yet complementary forms: a triangulated consensus protocol for modeling panel deformations, and a hinge consensus protocol for capturing panel-to-panel rotational dynamics. The triangular consensus protocol defines the target configuration within each individual triangular panel surface, ensuring that vertex coordination, driven by panel deformation, occurs strictly within the local 2D frame of the panel. In parallel, the hinge consensus protocol specifies the desired equilibrium angle between adjacent panels and drives the shared vertices toward configurations that satisfy this target angular relationship, thereby capturing the panel rotational motion during the reconfiguration process. 
The proposed FPCP, which emphasizes geometric coordination within each panel surface and at inter-panel hinges, fundamentally differs from existing projection-based consensus protocols.  Without accounting for the localized, structure-preserving requirements intrinsic to the origami reconfiguration, existing projection-based consensus protocols typically operate under global state constraints, such as projecting agents onto convex sets or maintaining predefined feasibility regions~\cite{chen2017global,lou2014approximate,nedic2010constrained,azizan2019distributed}.

Multiple factors, for example, the properties of the origami structure, tessellation patterns, and panel sizes, affect the consensus speed of each vertex within the origami structure network \cite{yan2025mechanical,li2019architected,zhang2024deployment}. {To} accurately capture the influence of all contributing factors, a data-fitting approach has been adopted to estimate the weighting factors embedded in the FPCP. Specifically, by measuring the coordinates of observed vertices in a time sequence, weights incorporated into the FPCP are estimated using the parameter optimization method to match the observation as closely as possible. The weighted FPCP (namely W-FPCP), as observed from existing weighted consensus protocols \cite{mukherjee2018robustness,park2016weighted}, outperforms FPCP with uniform weights when representing the origami reconfiguration process, where vertices exhibit diverse dynamic behaviors.

Our prior work in \cite{acc2025} developed a dynamic model to represent the origami panel expansion/contraction process based on the triangulated consensus protocol. {To} model the entire system evolution, this work includes both panel deformation and hinge rotation using the unified FPCP approach. Furthermore, rigorous proof of the convergence of FPCP is provided.
The contributions of this work are as follows: (1) design of FPCP that captures the panel deformation process occurring within their respective 2D surfaces, while also accounting for the rotation effects of the hinges; (2) proof of convergence of the proposed FPCP; and (3) development of a data-fitting approach to incorporate system parameters into the proposed FPCP.

% \S II, preliminaries including the properties of origami structures, 3D to 2D projection conversion, and network components for origami structures are introduced. We then discuss the limitations of the traditional consensus protocol and introduce the proposed FPCP in \S III. \S IV incorporates the weighting parameters into FPCP and adopts the new consensus protocol for the entire origami structure with convergence analysis and parameter estimation procedure. 
% Simulation results are provided in \S V, followed by conclusions in \S VI.
{Section II presents the preliminaries,} including the properties of origami structures, the 3D-to-2D projection, and the network components associated with origami structures. {Section III} discusses the limitations of the traditional consensus protocol and introduces the proposed FPCP. {Section IV} incorporates weighting parameters into the FPCP{, extends the protocol to} the entire origami structure{, and presents the} convergence analysis and parameter estimation procedure. {Section V provides simulation results, and Section VI concludes the paper.}

\section{Preliminaries}\label{sec:Pre}
In this section, we provide a concise overview of origami structures and their graph representation.
\subsection{Origami Structure Properties}
Origami-inspired structures exhibit two types of motion primitives that define their mechanical behavior during the folding/unfolding process: panel deformation and panel rotation around hinges. These motions enable the structure to achieve complex folding and unfolding patterns while retaining certain levels of stability and adaptability.

The first type of movement, panel deformation, refers to the evolution of the panels to shrink or expand under external forces. This deformation is affected by material characteristics, geometric configuration, and boundary conditions of the panels. These factors dictate the flexibility and resilience of the structure, allowing it to adapt its shape in response to external stimuli while preserving its structural integrity.

The second type of motion, the rotation of the panels around the hinges, refers to the angular movement at the hinges that connect the panels; the angular motion is governed by the moments acting on the hinges. %, and its equilibrium angle ensures stability in different configurations. 
The equilibrium of these panel angles plays a critical role in determining how the structure transitions between folded and unfolded states. The interplay between the panel deformation and rotation defines the states of the overall origami structure during the reconfiguration process. %kinematics and mechanics of the origami structure.

%To model and replicate these movements, the framework of multiagent consensus protocol is introduced. In this approach, the panels are treated as agents, and the hinges represent their interactions or constraints. By carefully tuning the parameters of the consensus algorithm, the system's deformation and rotation process can be encoded. %and structural behavior, such as stiffness, flexibility, and angular constraints.

%This modeling approach allows for the accurate mimicry of origami structure properties, providing a means to design and control such systems with high precision. In the subsequent sections, we will discuss how the parameters within the consensus algorithm are chosen and optimized to reflect the desired properties of the material. This highlights the utility of consensus protocol in capturing the complex mechanics of origami-inspired structures for various applications.

\subsection{Graph Representation for Origami Structures}
An origami configuration can be modeled as a planar straight-line graph, where edges represent the creases and vertices represent the points where these creases intersect. The complete set of all creases in an origami structure, typically visualized in a single diagram, is called the crease pattern and is denoted as $\mathcal{C}$. The crease pattern $\mathcal{C}$ serves as a two-dimensional blueprint for the folding process, and it encodes the geometric information necessary for constructing the three-dimensional (3D) origami structure from a two-dimensional (2D) sheet. 

Consider a multi-agent system, where ${\mathbf{x}}_i(t) \in \mathbb{R}^3$ denotes the coordinate of agent $i$ at time $t$. %{In the present framework, $t$ denotes a continuous-time evolution parameter associated with the consensus dynamics. Although it is not explicitly identified with physical actuation time, it serves as a time-like variable that parameterizes the progression of the system during the reconfiguration process.} 
In our setting, each agent $i$ represents a vertex in a triangulated configuration consisting of three vertices. An origami structure can thus be regarded as a set of such triangulated configurations, denoted as $\mathcal{G} = (\mathcal{V}, \mathcal{E}) $, where $\mathcal{G}$ is a graph representing the connections between vertices corresponding to the creases of the origami,  $\mathcal{V} = \{1, 2, \ldots, n\} $  is the set of vertices representing the agents, and $  \mathcal{E} \subseteq \mathcal{V} \times \mathcal{V} $  is the set of edges representing the incident relation between agents, e.g., $\{i, j\}\in \mathcal{E}$ for agents $i$ and $j$. 
An example of the graph representation of Miura-ori pattern origami is shown in Fig. \ref{f:muri-ex}.

\begin{figure}[h!]
	\centering
 	\includegraphics[scale=0.08] 	
  {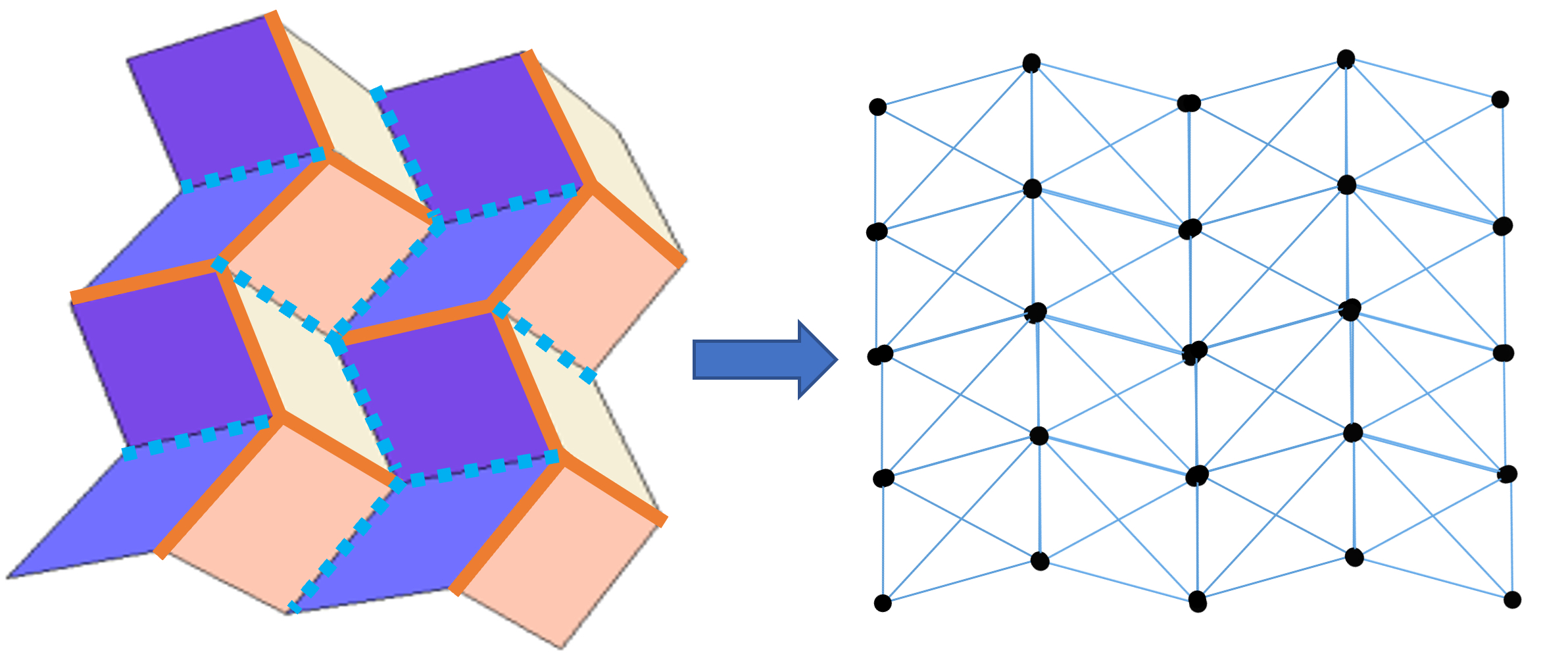}
  \caption{An example of Miura-ori origami structure and its graph representation, {where creases and crossbars represent edges, and nodes represent intersections of creases and crossbars}; the ``crossbars'' decompose the quadrilateral into triangular sets and are used to (subsequently) capture the bending behavior of each panel.}
	\label{f:muri-ex}
 \end{figure}

\subsection{3D to 2D Projection}
Projection from 3D space to a 2D plane is used in this work when modeling the system dynamics; 
as such, in this section we introduce the coordinate conversion for mapping a point from 3D to 2D. 
The projection of points in a 3D space into a 2D plane is a fundamental operation in many geometric 
and computational problems.

{
To define the projection, we introduce two orthonormal vectors 
$\mathbf{w}_1,\mathbf{w}_2 \in \mathbb{R}^{3}$ that span the local 2D plane. 
These vectors form a basis for the plane onto which the 3D coordinates are projected.
}
Together, these two normalized and mutually orthogonal vectors form the basis of the 2D plane, as well as the projection conversion matrix $M\in \mathbb{R}^{2\times3}$, written as
\begin{align}\label{e:proj_m_def}
    M =& \begin{bmatrix}
        \mathbf{w}_{1} & \mathbf{w}_{2}
        \end{bmatrix} ^{\top}.
\end{align}    
It thus follows that $MM^{\top}=I_2$, where $I_2\in\mathbb{R}^{2\times2}$ is the $2 \times 2$  identity matrix.
The projection conversion matrix constructed from these vectors serves as the foundation for mapping 3D points to the 2D plane, expressed as
\begin{equation} \label{e:3to2}
    \tilde{\mathbf{x}}_{\iota} = M \mathbf{x}_{\iota},
\end{equation}
where a 3D coordinate $\mathbf{x}_{\iota}\in \mathbb{R}^{3}$ is projected into a 2D plane with a coordinate $\tilde{\mathbf{x}}_{\iota}\in \mathbb{R}^{2}$.

\section{Frame Projection Consensus Protocol}\label{sec:triangulated_consensus}
Consensus protocol is widely utilized in distributed multi-agent systems to achieve agreement among agents on a common state. Consider a network defined by a vertex set $\mathcal{V}$ and edge set $\mathcal{E}$. The standard consensus protocol governing the dynamics of a system with $n$ agents can be expressed as
\begin{equation}\label{e:tradi_consensus}
   \dot{\mathbf{x}}(t)=-L(\mathcal{G}) \, \mathbf{x}(t),
\end{equation}
where $\mathbf{x}(t)=[x_1(t),x_2(t),\ldots,x_n(t)]^\top$ represents the states of the $n$ agents, and $L(\mathcal{G})$ denotes the graph Laplacian associated with the network. Extending this concept to formation control, specifically for a triangular formation defined by the complete graph ${\cal K}_3$, the consensus algorithm takes the form
\begin{equation}\label{e:formation}
\dot{\mathbf{x}}(t) = -L({\mathcal{K}}_3)(\mathbf{x}(t)-\mathbf{r}),
\end{equation}
where $ \mathbf{r} = [r_{ij}]$, for $(i, j) \in \mathcal{E}$, specifies the desired relative positions between the agents. Under this protocol, the agents' states asymptotically converge to the target formation, achieving $\lim_{t \to \infty} (x_i(t) - x_j(t)) = r_{ij}$, for every edge $(i, j) \in \mathcal{E}$.

However, when employing the traditional consensus protocol for guiding vertices toward a geometric formation, vertex coordination occurs independently along each axis in the 3D space. Consequently, such coordination does not inherently restrict panel deformation to the intended 2D panel surfaces,  failing to preserve the structural integrity during origami reconfiguration. Considering the limitations of traditional consensus protocols, we propose the FPCP that identifies the corresponding 2D projection planes where deformation or rotation of panels occurs. The consensus-based formation law is then operated within the correct geometric surface to preserve structural integrity during the origami reconfiguration.

\subsection{Triangulated Consensus Protocol}
This section introduces the triangulated consensus protocol to model the panel deformation process.
We start by assuming a connected triangular graph $ \mathcal{G} = (\mathcal{V}, \mathcal{E}) $ on three vertices ($n=3$) with coordinates ${\mathbf{x}}_{i, j, k}=[\mathbf{x}^\top_i,\mathbf{x}^\top_j,\mathbf{x}^\top_k]^\top$, where $\mathbf{x}_{\iota}\in \mathbb{R}^3$, and $\iota=i,j,k$ designate the indices of these vertices. 
To project a 3D coordinate to a 2D triangular panel surface, we employ the projection conversion matrix defined in \eqref{e:proj_m_def}. In this case, the projection conversion matrix is determined by 
\begin{equation}\label{eqn:project}
    M_{\text{panel}} = \begin{bmatrix} \frac{\mathbf{p}_{j, i}}{||\mathbf{p}_{j, i}||}  & \frac{(\mathbf{p}_{j, i} \times \mathbf{p}_{k, i}) \times \mathbf{p}_{j, i}}{||(\mathbf{p}_{j, i}\times \mathbf{p}_{k, i}) \times \mathbf{p}_{j, i}||}  \end{bmatrix} ^{\top} ,
\end{equation}
where $M_{\text{panel}}\in \mathbb{R}^{2\times 3}$, $\mathbf{p}_{\iota, \kappa} = \mathbf{x}_{\iota} - \mathbf{x}_{\kappa}$, $\iota,\kappa=i,j,k,\,\iota\neq\kappa$.
Then the 2D coordinates of the vertices in individual triangular panel surfaces, denoted as $\tilde{\mathbf{x}}_{i, j, k}\in \mathbb{R}^6$, can be computed via
\begin{equation}
\tilde{\mathbf{x}}_{i, j, k}=\tilde{M}_{\text{panel}} \, \mathbf{x}_{i, j, k} ,
\end{equation}
where $\tilde{M}_{\text{panel}} = I_{3} \otimes M_{\text{panel}}$,  $I_{3}\in \mathbb{R}^{3}$ is the $3 \times 3$ identity matrix, and `$\otimes$' denotes the Kronecker product.

The traditional consensus protocol is now applied to the projected coordinates on the 2D plane, with respect to the desired planar triangular configuration, denoted by $\tilde{\mathbf{r}}^{\text{panel}\top}_{i, j, k}=[\tilde{\mathbf{r}}^{\text{panel}\top}_i,\tilde{\mathbf{r}}^{\text{panel}\top}_j,\tilde{\mathbf{r}}^{\text{panel}\top}_k]^{\top}$, with $\tilde{\mathbf{r}}^{\text{panel}}_\iota\in \mathbb{R}^{2}$, $\iota=i,j,k$. Furthermore, the relative vector $\tilde{\mathbf{r}}^{\text{panel}}_j-\tilde{\mathbf{r}}^{\text{panel}}_i$ is aligned with $\mathbf{p}_{i, j}$ such that
\begin{equation}\label{eq:vector_r}
\tilde{\mathbf{r}}^{\text{panel}}_j-\tilde{\mathbf{r}}^{\text{panel}}_i=\alpha M_{\text{panel}}\mathbf{p}_{i, j},
\end{equation}
where $\alpha\in\mathbb{R}$ is a scalar. Lastly, we assume that $\tilde{\mathbf{r}}^{\text{panel}}_i$ overlaps with $M_{\text{panel}}\mathbf{x}_i$, i.e., $\tilde{\mathbf{r}}^{\text{panel}}_i=M_{\text{panel}}\mathbf{x}_i$. {The alignment and overlap simply choose vertex $i$ as the local reference point, so that the consensus law drives relative deformation rather than introducing an arbitrary in-plane shift.} 

Next, we apply the consensus protocol, in the context of formation control, to the projected coordinates of the origami panel. {For a triangular (complete) graph $K_3$, the consensus interaction among the vertices is represented through the graph Laplacian $L({\cal K}_3)$}. The consensus dynamics on the projected $2$D coordinates can therefore be written as
\begin{equation}
    \dot{\tilde{\mathbf{x}}}_{i, j, k} = -L_2 (\tilde{\mathbf{x}}_{i, j, k} - \tilde{\mathbf{r}}^{\text{panel}}_{i, j, k}),
\end{equation}
where $L_2 = L({\cal K}_3) \otimes I_{2}$. The velocity prescribed by the 2D consensus is then lifted back to the 3D space as,
\begin{equation}
    \dot{\mathbf{x}}_{i, j, k} = \tilde{M}_{\text{panel}}^{\top} \,\dot{\tilde{\mathbf{x}}}_{i, j, k}.
\end{equation}
In summary, the dynamics of the vertices based on the triangulated consensus assumes the form $\dot{\mathbf{x}}_{i, j, k} = f(\mathbf{x}_{i, j, k})$, where,
\begin{equation}\label{e:proj}
    f_{\text{panel}}(\mathbf{x}_{i, j, k}) = - \tilde{M}_{\text{panel}}^{\top} L_2 (\tilde{M}_{\text{panel}} \mathbf{x}_{i, j, k} - \tilde{\mathbf{r}}^{\text{panel}}_{i, j, k}) .
\end{equation}

Transformation from 3D to 2D coordinates using the projection matrix converts the origami panel representation from a global frame to its local frame. This approach enables capturing dynamic features that cannot be addressed using the traditional consensus control, namely, driving to a target formation specified within a 2D plane. 

{To} illustrate the working mechanism of the triangulated consensus protocol and its difference from traditional consensus, a comparative example is provided in Fig. \ref{fig:origami_demo1}.  Consider a triangulated graph in Fig. \ref{fig:origami_demo1}(a), consisting of a target formation (blue) and the initial configuration of the graph (red panel on the left side at $t=0$). According to the projection conversion matrix in \eqref{eqn:project}, the orthogonal vectors $\mathbf{w}_1$ and $\mathbf{w}_2$ are defined in Fig. \ref{fig:origami_demo1}(b). Using the triangulated consensus protocol, nodes assume the target formation on the original panel surface without changing the panel orientation, as shown in Fig. \ref{fig:origami_demo1}(c). However, when applying the traditional consensus protocol in \eqref{e:formation}, the coordination process is accompanied by a 3D rotation before converging to the target shape shown in Fig. \ref{fig:origami_demo1}(d). {Specifically,} the unexpected 3D rotation using the traditional consensus protocol fails to capture this important aspect of the origami dynamics during panel expansion/contraction.
\begin{figure}[h!]
	\centering
 	\includegraphics[scale=0.46] 	
  {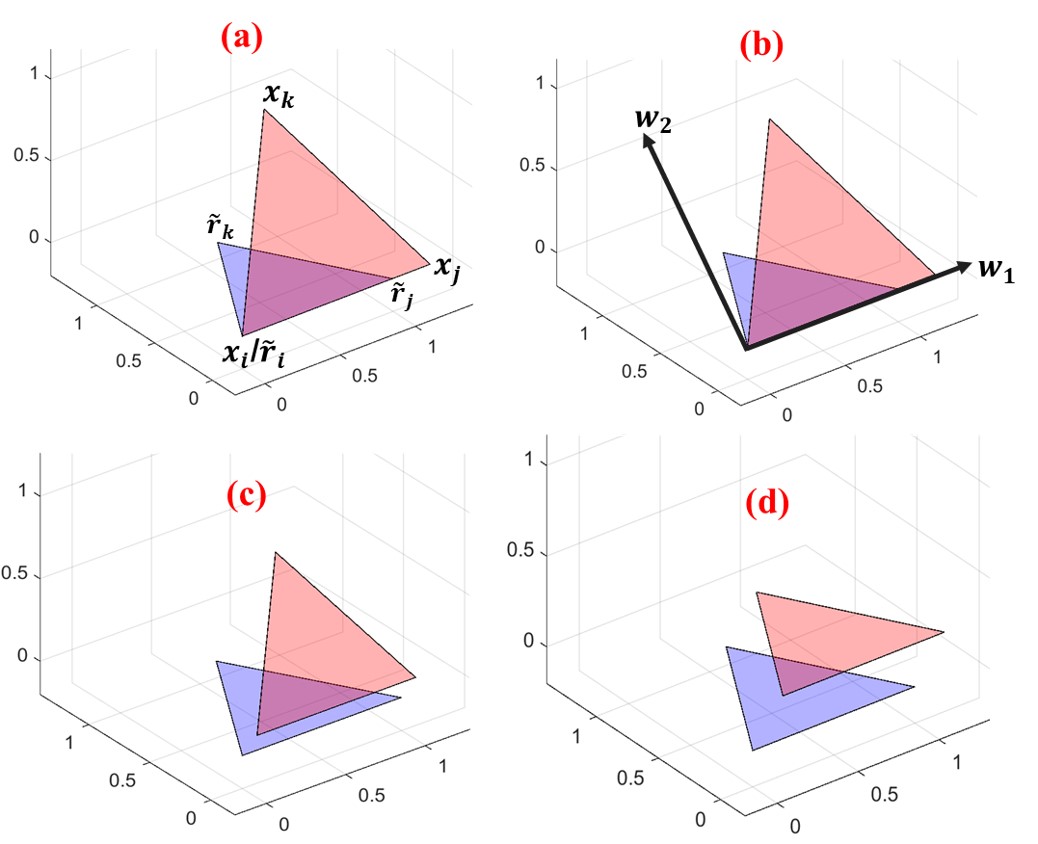}
  \caption{ {An example to model origami panel deformation by applying triangulated consensus protocol (a)$\rightarrow$(c), where consensus is applied to the local $(x_i,x_j,x_k)$ plane; and traditional consensus protocol (a)$\rightarrow$(d), where consensus is applied to the global 3D space.}}
	\label{fig:origami_demo1}
 \end{figure}

\subsection{Hinge Consensus Protocol}\label{ss:hinge_consensus}
Hinge consensus aims to guide the structure in attaining the desired angles between two adjacent panels. 
The FPCP is then applied to two projected surfaces defined by four vertices, denoted $i$, $j$, $k$, and $l$. As shown in {Fig. \ref{fig:hinge_demo}}, the system configuration can be visualized as a diamond. The angle between two panels, composed of $\{l,k,j\}$ and $\{i,k,j\}$, is denoted by $\phi\in (-\pi, \pi]$. {To} drive the rotation of the two panels toward the equilibrium angle $\phi$, two projection surfaces, composed of $\{i,k,l\}$ and $\{i,j,l\}$, are used to construct the two projection conversion matrices, denoted by $M_{\text{hinge1}}$ and $M_{\text{hinge2}}$, respectively. {These projection surfaces are chosen so that the hinge geometry is explicitly retained in the local construction: the hinge is the shared edge of the two adjacent panels, and the relevant motion is their relative rotation about that common edge. Accordingly, the selected surfaces serve as auxiliary geometric frames anchored to the hinge configuration, allowing the planar target formations to encode the desired inter-panel angle rather than an arbitrary in-plane deformation.} These matrices are defined as
\begin{subequations}\label{eq:hinge_map}
\begin{align}
    M_{\text{hinge1}} &= \begin{bmatrix}
        \frac{\mathbf{p}_{k, l}}{\|\mathbf{p}_{k, l}\|} & \frac{\mathbf{p}_{j, k} \times \mathbf{p}_{k, l}}{\|\mathbf{p}_{j, k} \times \mathbf{p}_{k, l}\|}
    \end{bmatrix}{^\top}, \\
    M_{\text{hinge2}} &= \begin{bmatrix}
        \frac{\mathbf{p}_{j, i}}{\|\mathbf{p}_{j, i}\|} & \frac{\mathbf{p}_{i, j} \times \mathbf{p}_{k, j}}{\|\mathbf{p}_{i, j} \times \mathbf{p}_{k, j}\|}
    \end{bmatrix}{^\top},
\end{align}
\end{subequations}
where $\mathbf{p}_{\iota, \kappa} = \mathbf{x}_{\iota} - \mathbf{x}_{\kappa}$, $\iota,\kappa=i,j,k,l,\,\iota\neq\kappa$. 
{The lifting in \eqref{eq:hinge_map} should not be interpreted as a unique inverse from arbitrary 2D motions to arbitrary 3D motions. Rather, for any local frame used in the proposed protocol, let $M_{\Gamma}\in\{M_{\mathrm{panel}},M_{\mathrm{hinge1}},M_{\mathrm{hinge2}}\}$. Once \(M_{\Gamma}\) is determined by the current local configuration, the projected consensus law prescribes an in-plane velocity in the corresponding local coordinates, and \(\tilde M_{\Gamma}^{\top}\) embeds this velocity back into the associated local subspace of \(\mathbb{R}^3\). Any component in \(\ker(M_{\Gamma})\), corresponding to the out-of-plane direction of the local frame, is therefore intentionally set to zero at this local update stage. Global compatibility across the origami structure is then enforced through the aggregation of local contributions at shared vertices in the full W-FPCP model discussed next.}
% {The lifting in \eqref{eq:hinge_map} should not be interpreted as a unique inverse from arbitrary 2D motions to arbitrary 3D motions. Rather, once the local panel frame is determined by $M_{\mathrm{panel}}$ at the current configuration, the projected consensus law prescribes an in-plane velocity in the corresponding panel coordinates, and $\tilde M_{\mathrm{panel}}^\top$ embeds this velocity back into the associated panel subspace of $\mathbb{R}^3$. Since $M_{\mathrm{panel}}M_{\mathrm{panel}}^\top=I_2$, this map provides the canonical in-plane 3D representative relative to the chosen panel frame, while any out-of-plane component is intentionally excluded at this local update stage. Global compatibility across the origami structure is then enforced through the aggregation of local contributions at shared vertices in the full W-FPCP model discussed next.}
\vspace{-0.3cm}
\begin{figure}[h!]
	\centering
 	\includegraphics[scale=0.5] 	
  {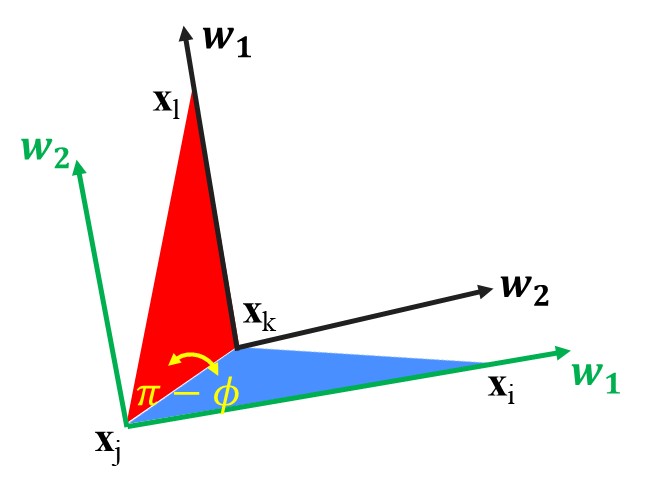}
  \vspace{-0.3cm}
  \caption{Illustration of settings in the hinge consensus protocol.}
	\label{fig:hinge_demo}
 \end{figure}

% The system's configuration can be visualized as a diamond, with the angles between the panels represented as triangular relationships between the vertices. An example is illustrated in Fig. \ref{fig:ex_hinge}, where the hinge, initially positioned at 90 degrees, attempts to return to its neutral angle of 180 degrees. There will exist $\{i,k,l\}$ side and $\{i,j,l\}$ side. Each side of the diamond configuration corresponds to a panel, and its equilibrium can be analyzed in a 2D formation. This will still follow the same idea as \eqref{e:proj}, expressed as
% \begin{subequations}
% \begin{align}
%     f(\mathbf{x}_{i, k, l}) = - \tilde{M}_{\text{hinge1}}^{\top} L_2 (\tilde{M}_{\text{hinge1}} \mathbf{x}_{i, k, l} - \tilde{\mathbf{r}}_{i, k, l}^{\text{hinge1}}) \\
%     f(\mathbf{x}_{i, j, l}) = - \tilde{M}_{\text{hinge2}}^{\top} L_2 (\tilde{M}_{\text{hinge2}} \mathbf{x}_{i, j, l} - \tilde{\mathbf{r}}_{i, j, l}^{\text{hinge2}})
% \end{align}
% \end{subequations}

After defining the two projection surfaces and their corresponding projection conversion matrices, the next step is to define the target formation, denoted by $\tilde{\mathbf{r}}_{i, k, l}$ and $\tilde{\mathbf{r}}_{i, j, l}$, respectively, creating a triangular shape with the target angle $\phi$ on each projection surface. To do so, we first use the equilibrium angle between the two panels to construct a rotation matrix from Rodrigues' rotation formula, expressed as
\begin{equation}
\begin{split}
\Phi_\phi &= I_3 \cos{\phi} + \left[\frac{\mathbf{q}_{k,j}}{\left\| \mathbf{q}_{k,j}\right\|} \right]_\times \sin{\phi} \\[6pt]
          &\quad + \left(\frac{\mathbf{q}_{k,j}}{\left\| \mathbf{q}_{k,j}\right\|} \right)\left(\frac{\mathbf{q}_{k,j}}{\left\| \mathbf{q}_{k,j}\right\|} \right)^\top (1-\cos{\phi}),
\end{split}
\end{equation}
where {$[\cdot]_{\times}$} denotes the skew‐symmetric matrix and $\mathbf{q}_{\iota,\kappa} \in \mathbb{R}^{3}$ is a desired vector from vertex $\iota$ to $\kappa$ at degree $\phi = 0$ {for simplicity to represent the structure on a 2D plane to allow for $\phi\in (-\pi, \pi]$}.

The positions of the target formation in the 2D projection planes are now defined as
\begin{subequations}
\begin{align}
    \begin{split}
\tilde{\mathbf{r}}_{i}^{\text{hinge1}}
&= 
  \left[
    \left\langle \frac{\mathbf{q}_{k, l}}{\|\mathbf{q}_{k, l}\|},\,\Phi_\phi\mathbf{q}_{i, j} - \mathbf{q}_{k, j} \right\rangle
  \right.\\
&\quad
  \left.
    \left\langle \frac{-\mathbf{q}_{k,j} \times \mathbf{q}_{k,l}}{\|-\mathbf{q}_{k,j} \times \mathbf{q}_{k,l}\|},\,\Phi_\phi\mathbf{q}_{i, j} - \mathbf{q}_{k, j} \right\rangle
  \right]^{\!\top},
\end{split}\\
    \tilde{\mathbf{r}}_{k}^{\text{hinge1}} &= \begin{bmatrix}
        0 & 0
    \end{bmatrix}^{\top}, \\
    \tilde{\mathbf{r}}_{l}^{\text{hinge1}} &= \begin{bmatrix}
        -\|\mathbf{q}_{k, l}\| & 0
    \end{bmatrix}^{\top}, \\
    \tilde{\mathbf{r}}_{i}^{\text{hinge2}} &= \begin{bmatrix}
        -\|\mathbf{q}_{j, i}\| & 0
    \end{bmatrix}^{\top}, \\
    \tilde{\mathbf{r}}_{j}^{\text{hinge2}} &= \begin{bmatrix}
        0 & 0
    \end{bmatrix}^{\top}, \\
    \tilde{\mathbf{r}}_{l}^{\text{hinge2}} &= \begin{bmatrix}
    -\left\langle \frac{\Phi_\phi\mathbf{q}_{i, j}}{\|\Phi_\phi\mathbf{q}_{i, j}\|}, \mathbf{q}_{l, j} \right\rangle & \left\langle \frac{\Phi_\phi\mathbf{q}_{i, j} \times \mathbf{q}_{k, j}}{\| \Phi_\phi\mathbf{q}_{i, j} \times \mathbf{q}_{k, j} \|}, \mathbf{q}_{l, j} \right\rangle
    \end{bmatrix}^{\top}.
\end{align}
\end{subequations}
{where $\langle \cdot, \cdot \rangle$ denotes the inner product between two vectors. Each vector defines the position of an element in the target formation, expressed in the local 2D plane corresponding to the side view of the two panels for the given angle $\phi$.}

Then, the target triangular formation is determined by
\begin{subequations} \label{e:hinge_desired}
\begin{align}
    \tilde{\mathbf{r}}_{i, k, l}^{\text{hinge1}} &= \begin{bmatrix}
         \tilde{\mathbf{r}}_{i}^{\text{hinge1}\top}  & \tilde{\mathbf{r}}_{k}^{\text{hinge1}\top} & \tilde{\mathbf{r}}_{l}^{\text{hinge1}\top}
    \end{bmatrix}^{\top}, \\
    \tilde{\mathbf{r}}_{i, j, l}^{\text{hinge2}} &= \begin{bmatrix}
         \tilde{\mathbf{r}}_{i}^{\text{hinge2}\top} & \tilde{\mathbf{r}}_{j}^{\text{hinge2}\top} &  \tilde{\mathbf{r}}_{l}^{\text{hinge2}\top} 
    \end{bmatrix}^{\top}.
\end{align}
\end{subequations}
The definition of the target formation $\tilde{\mathbf{r}}^{\text{hinge1}}_{i, k, l}$ ensures that the angle between vectors $\tilde{\mathbf{r}}^{\text{hinge1}}_{i}-\tilde{\mathbf{r}}^{\text{hinge1}}_{{k}}$ and $\tilde{\mathbf{r}}^{\text{hinge1}}_{l}-\tilde{\mathbf{r}}^{\text{hinge1}}_{{k}}$ is the desired angle $\phi$. Meanwhile, the definition of target formation $\tilde{\mathbf{r}}^{\text{hinge2}}_{i, j, l}$ ensures that the angle between vectors $\tilde{\mathbf{r}}^{\text{hinge2}}_{i}-\tilde{\mathbf{r}}^{\text{hinge2}}_{{j}}$ and $\tilde{\mathbf{r}}^{\text{hinge2}}_{l}-\tilde{\mathbf{r}}^{\text{hinge2}}_{{j}}$ is the desired angle $\phi$ as well.

Next, we apply the same consensus protocol in each projection plane, written as
\begin{subequations}\label{e:hinge_con}
\begin{align}
    f_{\text{hinge1}}(\mathbf{x}_{i, k, l}) = - \tilde{M}_{\text{hinge1}}^{\top} L_2 (\tilde{M}_{\text{hinge1}} \mathbf{x}_{i, k, l} - \tilde{\mathbf{r}}_{i, k, l}^{\text{hinge1}}), \\
    f_{\text{hinge2}}(\mathbf{x}_{i, j, l}) = - \tilde{M}_{\text{hinge2}}^{\top} L_2 (\tilde{M}_{\text{hinge2}} \mathbf{x}_{i, j, l} - \tilde{\mathbf{r}}_{i, j, l}^{\text{hinge2}}).
\end{align}
\end{subequations}
The formation coordination in both projection planes will drive the two panels $\{l,k,j\}$ and $\{i,k,j\}$ toward the target angle $\phi$. One example is shown in Fig. \ref{fig:ex_hinge}, where the hinge, initially positioned at 90 degrees, attempts to converge to its equilibrium angle of 180 degrees as $t\rightarrow\infty$.

%The desired reference position is inherently more complex than the corresponding triangulated consensus protocol due to the spatial considerations at the hinges. Initially, the values must account for the hinges being in a flat configuration, which is represented in a 2D plane. If the reference positions in this flat configuration are denoted by $\tilde{\mathbf{s}}_{\iota}$ for $\iota = i, j, k, l$, then the relative position between two vertices is given by $\tilde{\mathbf{q}}_{\iota, \kappa} = \tilde{\mathbf{s}}_{\iota} - \tilde{\mathbf{s}}_{\kappa}$.
%
\begin{figure}[h!]
    \centering
    \includegraphics[width=1\linewidth]{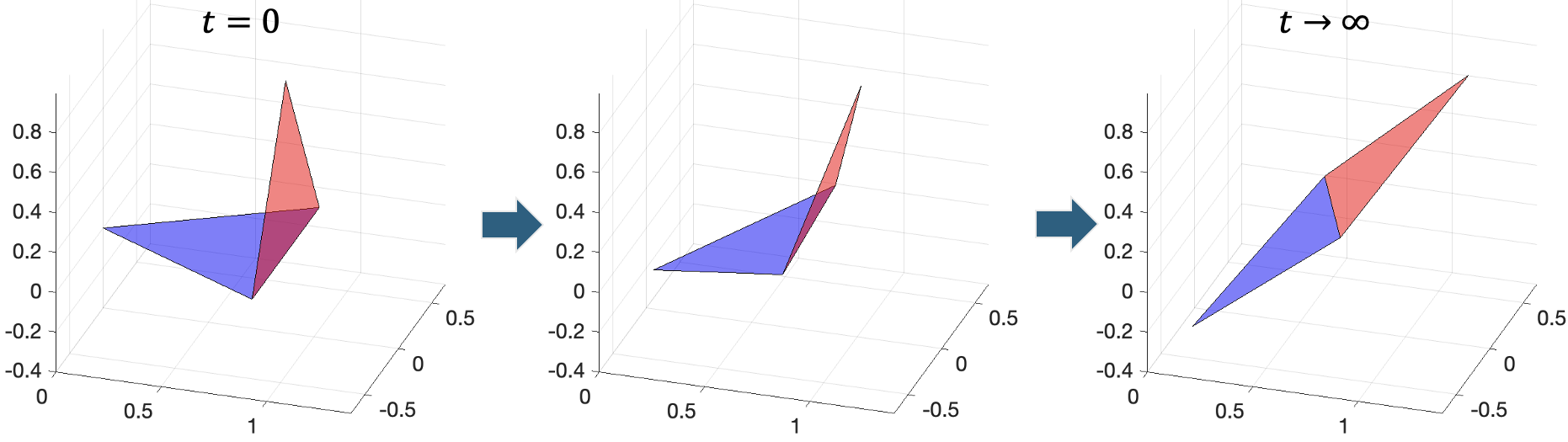}
    \caption{An example of applying the hinge consensus protocol to origami panel rotations.}
    \label{fig:ex_hinge}
\end{figure}

\section{Weighted Frame Projected Consensus for Origami Dynamics}\label{s:weighted}
% \subsection{Weighted Frame Projected Consensus Protocol}
The triangulated consensus and hinge consensus protocols introduced above are applicable to a single panel deformation and panel rotations around one hinge, respectively. In meantime, factors that influence the speed of each vertex during the reconfiguration have not been considered in these protocols. The weighted frame projection consensus protocol (W-FPCP) is a framework designed to apply the FPCP to the entire origami system while incorporating parameters that capture each panel's deformation properties and
enhance the modeling fidelity of the protocol. % and structural characteristics. 
{In particular, the weighted consensus protocol assigns effective interaction coefficients to the edges of the triangular graph to improve trajectory reproduction while accounting, in an aggregate phenomenological sense, for effects associated with strain-energy-related behavior, stiffness variation, hinge resistance, and material heterogeneity. These weights are not introduced as direct one-to-one physical stiffness or damping parameters with prescribed mechanical units derived from a first-principles model; instead, they are treated as fitted effective parameters within the proposed consensus-based reconfiguration model.} By embedding these parameters into the reconfiguration process, the model not only achieves higher accuracy and adaptability, but also allows for the representation of diverse origami configurations. %This weighted approach serves as a foundation for exploring how material properties influence the overall dynamics and behavior of the structure during its reconfiguration. 
The weighted Laplacian, denoted $\Omega({\cal K}_3)\in\mathbb{R}^{3 \times 3}$, for the triangular graph is used to replace the unweighted Laplacian $L({\cal K}_3)$, expressed as
\begin{equation}
\Omega({\cal K}_3)= \begin{bmatrix}
\omega_{i,j}+\omega_{i,k} &-\omega_{i,j} &-\omega_{i,k}\\
-\omega_{i,j} &\omega_{i,j}+\omega_{j,k} &-\omega_{j,k}\\
-\omega_{i,k}& -\omega_{j,k} &\omega_{i,k}+\omega_{j,k}
\end{bmatrix} ,
\end{equation}
where $\omega_{\iota,\kappa}={\omega_{\kappa,\iota}}\ge 0$, $\iota,\kappa=i,j,k$, $\iota\neq\kappa$, are weights of the undirected graph corresponding to vertices $\iota$ and $\kappa$. %{The consensus weights can be interpreted as effective interaction coefficients that reflect structural properties such as stiffness and hinge resistance within the origami system.} 
{To} extend the weighted Laplacian to a triangular graph on a 2D plane, analogous to $L_2$ in \eqref{e:proj}, we define $\Omega_{i,j,k}\in \mathbb{R}^{6\times6}$, written as 
\begin{equation}\label{e:weights}
\begin{aligned}
\!\!\Omega_{i,j,k}= &\left[\begin{matrix}
\omega_{i,j}+\omega_{i,k} & 0 &-\omega_{i,j} \\
0 & \gamma_{i,j}+\gamma_{i,k}&0 \\
-\omega_{i,j} &0 &\omega_{i,j}+\omega_{j,k}\\
0 &-\gamma_{i,j} &0 \\
-\omega_{i,k} &0 & -\omega_{j,k} \\
0 & -\gamma_{i,k} & 0\\ 
\end{matrix}\right.\\
&\qquad\;\;
\left.\begin{matrix}
0 &-\omega_{i,k} &0\\
-\gamma_{i,j} &0 &-\gamma_{i,k}\\
0 &-\omega_{j,k}&0\\
\gamma_{i,j}+\gamma_{j,k}&0 &-\gamma_{j,k}\\
0 & \omega_{i,k}+\omega_{j,k} &0\\
-\gamma_{j,k} &0 & \gamma_{i,k}+\gamma_{j,k}
\end{matrix}\right] ,
\end{aligned}
\end{equation}
where elements $\omega_{\iota,\kappa}\ge0$ and $\gamma_{\iota,\kappa}\ge0$, $\iota,\kappa=i,j,k$, $\iota\neq\kappa$, denote weights along the two coordinates on a 2D plane, respectively. {Since the graph is undirected, we assume $\omega_{\iota,\kappa}=\omega_{\kappa,\iota}$ and $\gamma_{\iota,\kappa}=\gamma_{\kappa,\iota}$ so that $\Omega_{i,j,k}$ is symmetric and positive semidefinite.} The approach to determining these weights is addressed in \S V to capture variations of panel properties. With the introduced weighted Laplacian representing weights for a triangular graph on a 2D plane, the frame projected consensus is then updated via
\begin{equation}\label{eqn:weighted_consensus}
    \dot{\mathbf{x}}_{i, j, k} = -\tilde{M}_{i,j,k}^{\top} \Omega_{i,j,k}\tilde{M}_{i,j,k} \left(\mathbf{x}_{i, j, k} - \tilde{M}_{i,j,k}^ {\top}\tilde{\mathbf{r}}_{i, j, k}\right),
\end{equation}
where $\tilde{M}_{i,j,k} = I_{3} \otimes M_{i,j,k} \in\mathbb{R}^{6\times9}$ and $M_{i,j,k}\in \mathbb{R}^{2\times3}$ is the projection conversion matrix for the 2D panel composed by vertices $\{i,j,k\}$.

\subsection{W-FPCP for Entire Origami Structure}
The W-FPCP is applied individually to each panel deformation or rotation around the hinge above, ensuring that the dynamics of local interactions are accurately captured. {By }accounting for the {cumulative} contributions of each panel's deformation and rotation about every hinge, we synthesize the collective dynamics of the entire origami structure. The resulting comprehensive model accounts for the interplay between panels and hinges, as well as their shared constraints. %This approach ensures that the structure's overall behavior emerges from its constituent parts' coordinated contributions. 
Let the index of $m$ panels and $h$ hinges be denoted by $\Gamma = 1, ..., m+2h$ {as each hinge consensus protocol requires  consensus coordination twice as discussed in Section~\ref{ss:hinge_consensus}}; the overall system dynamics is now written as 
\begin{equation}\label{e:weighted_alltogether}
    \dot{\mathbf{x}} = \sum_{\Gamma=1}^{m+2h} (S_{i,j,k} \otimes I_{3}) \, \dot{\mathbf{x}}^{(\Gamma)}_{i, j, k},
\end{equation}
where $\mathbf{x}\in\mathbb{R}^{3\nu}$ is the stacked state vector for $\nu$ vertices within the entire origami structure, $S_{i,j,k} \in \mathbb{R}^{\nu \times 3}$ is a selection matrix with all entries set to zero except $S_{i,j,k}(i,1)=S_{i,j,k}(j,2)=S_{i,j,k}(k,3)=1$. The selection matrix is required to align the local vertices $\mathbf{x}^{(\Gamma)}_{i, j, k}$ in every triangular graph with the corresponding vertices in the stacked state vector $\mathbf{x}$. For each formation (panel or hinge) $\Gamma = 1, \dots, m+2h$, the corresponding graph is $\mathcal{G}^{(\Gamma)} = (\mathcal{V}^{(\Gamma)}, \mathcal{E}^{(\Gamma)})$, where $\mathcal{V}^{(\Gamma)} = \{i, j, k\}$ denotes the three vertices of the triangular graph and $\mathcal{E}^{(\Gamma)}$ represents the edges that connect them.
The term $\dot{\mathbf{x}}^{(\Gamma)}_{i, j, k}$ follows the W-FPCP in \eqref{eqn:weighted_consensus} for panel $\Gamma$. Using \eqref{e:weighted_alltogether}, the dynamics of the entire origami structure can be determined by accumulating the contributions of each local set of triangular vertices.  {Because adjacent panels share common vertices in the aggregated formulation, the motion of each such vertex is influenced by all incident local contributions, which enforces kinematic consistency at the shared-vertex level throughout the evolution. As a result, the global structure connectivity is guaranteed by the the shared-vertex aggregation formulation.}

The traditional consensus protocol in 
 \eqref{e:tradi_consensus} operates as a linear system governed by a more streamlined relationship between the states and their evolution. However, after incorporating the projection conversion matrix, the dynamics becomes nonlinear, 
 introducing new types of dynamic phenomena. {T}o obtain a compact form of the corresponding
 system dynamics, a series of conversions is introduced below to transform the system dynamics into a state-space representation.

For every $\dot{\mathbf{x}}^{(\Gamma)}_{i, j, k}$, $\Gamma = 1, \dots, m+2h$, expressed in \eqref{e:weighted_alltogether}, we rewrite it as 
\begin{equation}\label{eqn:weighted_consensus2}
    \dot{\mathbf{x}}^{(\Gamma)}_{i, j, k} = -W_{i,j,k} \mathbf{x}_{i, j, k}^{(\Gamma)} - W_{i,j,k}\mathbf{b}_{i,j,k},
\end{equation}
where $W_{i,j,k}=\tilde{M}_{i,j,k}^{\top} \Omega_{i,j,k}\tilde{M}_{i,j,k}\in \mathbb{R}^{9\times 9}$, $\mathbf{b}_{i,j,k}=-\tilde{M}^{\top}_{i,j,k}\tilde{\mathbf{r}}_{i, j, k}^{(\Gamma)}\in \mathbb{R}^{9}$ {for all} $\mathcal{V}^{(\Gamma)} = \{i,j,k\}$. To account for the accumulative effect from every $\dot{\mathbf{x}}^{(\Gamma)}_{i, j, k}$, $\Gamma = 1, \dots, m+2h$, to the stacked state $\mathbf{x}$, $W_{i,j,k}$ is expanded to $\hat{W}_{i,j,k}\in \mathbb{R}^{3\nu\times 3\nu}$ by first decomposing $W_{i,j,k}$ into
\begin{equation}
W_{i,j,k} = \sum^{3}_{\iota = 1} \sum^{3}_{\kappa = 1} G_{\iota, \kappa} \otimes W_{\iota, \kappa},
\end{equation}
where selection $G_{\iota, \kappa} \in \mathbb{R}^{3\times3}$ with all entries set to zero except $G_{\iota, \kappa}(\iota, \kappa) = 1$ and  $W_{\iota, \kappa} \in \mathbb{R}^{3\times3}$ is the affect of consensus on agent $\iota$ on $\kappa$. From this, we expand the stacked state $\mathbf{x}$. {The local interaction matrix $W_{i,j,k}$ can be embedded into the global system dynamics using a Kronecker-product representation.} Thus, $W_{i,j,k}$ is expanded to $\hat{W}_{i,j,k}\in \mathbb{R}^{3\nu\times 3\nu}$, written as
\begin{equation}\label{e:all_weights}
\hat{W}_{i,j,k} = \sum_{\iota \in \mathcal{V}^{(\Gamma)}} \sum_{\kappa \in \mathcal{V}^{(\Gamma)}} \hat{G}_{\iota,\kappa} \otimes W_{\iota, \kappa},
\end{equation}
where $\hat{G}_{\iota, \kappa} \in \mathbb{R}^{\nu \times \nu}$ is a selection matrix with all entries set to zero except $\hat{G}_{\iota, \kappa}(\iota, \kappa) = 1$. Then $\mathbf{b}_{i,j,k}$ is expanded to $\hat{\textbf{b}}_{i,j,k}\in \mathbb{R}^{3\nu}$, written as 
\begin{equation}
\hat{\textbf{b}}_{i,j,k}=(S^{(\Gamma)}_{i,j,k} \otimes I_{3})\mathbf{b}_{i,j,k}.
\end{equation}
Through the expansion of matrices, the system dynamics in \eqref{e:weighted_alltogether} is written as
\begin{equation}\label{e:weighted_alltogether2}
    \dot{\mathbf{x}} = \sum_{\Gamma=1}^{m+2h} -\hat{W}_{i,j,k} {(}\mathbf{x} {+} \hat{\mathbf{b}}_{i,j,k}{)}.
\end{equation}
Next, by augmenting the system state vector to incorporate a constant term, denoted by $\hat{\mathbf{x}} =\begin{bmatrix} \mathbf{x} \\ 1 \end{bmatrix}\in \mathbb{R}^{3\nu+1}$, the system dynamics in \eqref{e:weighted_alltogether2} can be converted into a homogeneous form, 
\begin{equation}\label{e:homg}
    \dot{\hat{\mathbf{x}}} = A \hat{\mathbf{x}},
\end{equation}
where $$A = \sum_{\Gamma=1}^{m+2h} A_{i, j, k}^{(\Gamma)} \in \mathbb{R}^{(3\nu+1)\times (3\nu+1)}$$ and $$A_{i, j, k}^{(\Gamma)} = \begin{bmatrix} 
    -\hat{W}_{i,j,k} & -\hat{W}_{i,j,k} \hat{\mathbf{b}}_{i,j,k}\\  \mathbf{0}_{1\times 3\nu} & \mathbf{0}_{1\times 1} 
    \end{bmatrix},$$ with $\mathcal{V}^{(\Gamma)} = \{i, j, k\}$.
The time-varying linear representation of W-FPCP derived above now
facilitates its convergence analysis described next.

\subsection{Convergence of W-FPCP}
This section examines the convergence property of the origami dynamic model based on W-FPCP. {In particular, we start with the convergence proof of triangulated consensus in \eqref{e:proj} applied to a single panel and then extend to W-FPCP in \eqref{eqn:weighted_consensus} for one single panel. After that, we establish the convergence for the aggregated W-FPCP applied to the overall system under mild conditions.}
%establish a local Lyapunov stability result for equilibrium configurations induced by the proposed W-FPCP dynamics.}

\begin{lemma} \label{t:similar}
The matrices ${W_{i,j,k}=\tilde{M}_{i,j,k}^{\top} \Omega_{i,j,k} \tilde{M}_{i,j,k}}\in \mathbb{R}^{{9} \times 9}$ and $\Omega_{i,j,k} \in \mathbb{R}^{6 \times 6}$ defined in \eqref{e:weights} and \eqref{eqn:weighted_consensus}, respectively, share the same nonzero eigenvalues. In addition, $W_{i,j,k}$ is positive semidefinite { since $\Omega_{i,j,k}$ is positive semidefinite}.
\end{lemma}

\begin{proof}
This follows from a classical property of matrix similarity in reduced dimensions: for matrices $C \in \mathbb{R}^{n \times m}$ and $D \in \mathbb{R}^{m \times n}$, the nonzero eigenvalues of $CD \in \mathbb{R}^{n \times n}$ and $DC \in \mathbb{R}^{m \times m}$ are identical (including multiplicities).
Applying this to our case, we obtain
\begin{equation}
    {\lambda_{\text{nz}}}(\tilde{M}_{i,j,k}^{\top} \Omega_{i,j,k} \tilde{M}_{i,j,k}) = {\lambda_{\text{nz}}}(\tilde{M}_{i,j,k}\tilde{M}_{i,j,k}^{\top}\Omega_{i,j,k}  ),
\end{equation}
where $\lambda_{\text{nz}}(\cdot)$ denotes the set of nonzero eigenvalues. 
Since $\tilde{M} \tilde{M}^{\top} = I_6$ (using the definitions of $M$ and $\tilde{M}$), it follows that
\begin{equation}
    {\lambda_{\text{nz}}}(\tilde{M}_{i,j,k}^{\top} \Omega_{i,j,k} \tilde{M}_{i,j,k}) = {\lambda_{\text{nz}}}(\Omega_{i,j,k}).
\end{equation}
Thus, $\tilde{M}_{i,j,k}^{\top}\Omega_{i,j,k} \tilde{M}_{i,j,k}$ and $\Omega_{i,j,k}$ share the same set of
nonzero eigenvalues. In addition, as $\tilde{M}_{i,j,k}^{\top} \Omega_{i,j,k} \tilde{M}_{i,j,k}=(\tilde{M}_{i,j,k}^{\top} \Omega_{i,j,k} \tilde{M}_{i,j,k})^T$, $W_{i,j,k}$ is symmetric with non-negative eigenvalues. Thus, $W_{i,j,k}$ is positive semidefinite.
%Any additional eigenvalues introduced by the transformation will be zero due to the change in dimension.
\end{proof}

% {
% \begin{lemma}\label{l:same}
% Let \(W_{i,j,k}\in\mathbb R^{9\times 9}\) be the local interaction matrix
% associated with the vertex set \(\{i,j,k\}\), and let
% \(\hat W_{i,j,k}\in\mathbb R^{3\nu\times 3\nu}\) denote its embedding
% into the global coordinate space. Then
% $\lambda_{\mathrm{nz}}(\hat W_{i,j,k})=\lambda_{\mathrm{nz}}(W_{i,j,k}).
% $
% \end{lemma}}

% {
% \begin{proof}
% The embedding from \(W_{i,j,k}\) to \(\ehat W_{i,j,k}\) as defined in \eqref{e:weights} and \eqref{e:all_weights} only places the local interaction matrix into the rows and columns corresponding to the coordinates of vertices \(i,j,k\). All coordinates associated with vertices outside \(\{i,j,k\}\) are unaffected by this local interaction and therefore contribute zero rows and columns. Hence, after a simultaneous permutation of rows and columns that places the coordinates of vertices \(i,j,k\) first, there exists a permutation matrix \(P\) such that
% \[
% P\hat W_{i,j,k}P^\top
% =
% \begin{bmatrix}
% W_{i,j,k} & 0_{9\times(3\nu-9)}\\
% 0_{(3\nu-9)\times 9} & 0_{(3\nu-9)\times(3\nu-9)}
% \end{bmatrix}.
% \]
% Permutation similarity preserves the characteristic polynomial and hence
% the eigenvalues \cite[Sec.~1.3]{horn2013matrix}. Therefore,
% \[
% \det(\lambda I_{3\nu}-\hat W_{i,j,k})
% =
% \lambda^{3\nu-9}\det(\lambda I_9-W_{i,j,k}),
% \]
% so the embedding only appends \(3\nu-9\) additional zero eigenvalues.
% Thus the nonzero eigenvalues of \(\hat W_{i,j,k}\) and \(W_{i,j,k}\)
% coincide.
% \end{proof}}

\begin{corollary}\label{t:WequalOmega}
    The matrices $\Omega_{i,j,k}\in \mathbb{R}^{6 \times 6}$ and $\hat{W}_{i,j,k}\in \mathbb{R}^{{3}\nu \times {3}\nu}$ share the same set of nonzero eigenvalues.  In addition, $\hat W_{i,j,k}$ is positive semidefinite.
\end{corollary}
\begin{proof}
{
The embedding from \(W_{i,j,k}\) to \(\hat W_{i,j,k}\) as defined in \eqref{e:weights} and \eqref{e:all_weights} only places the local interaction matrix into the rows and columns corresponding to the coordinates of vertices \(i,j,k\).  Permutation similarity preserves the characteristic polynomial and hence the eigenvalues \cite[Sec.~1.3]{horn2013matrix}.  Therefore,
$\det(\lambda I_{3\nu}-\hat W_{i,j,k}) = \lambda^{3\nu-9}\det(\lambda I_9-W_{i,j,k})$.} we have ${\lambda_{\text{nz}}}(\hat{W}_{i,j,k})={\lambda_{\text{nz}}}({W}_{i,j,k})$. From Lemma \ref{t:similar}, we have ${\lambda_{\text{nz}}}(\tilde{M}_{i,j,k}^{\top} \Omega_{i,j,k} \tilde{M}_{i,j,k}) = {\lambda_{\text{nz}}}(\Omega_{i,j,k})$. As $W_{i,j,k}=\tilde{M}_{i,j,k}^{\top} \Omega_{i,j,k}\tilde{M}_{i,j,k}$, it thus follows that ${\lambda_{\text{nz}}}(\Omega_{i,j,k})={\lambda_{\text{nz}}}(\hat{W}_{i,j,k})$. As $\hat W_{i,j,k}$ is symmetric with non-negative eigenvalues, $\hat W_{i,j,k}$ is positive semidefinite.
\end{proof}

{
Next, we first examine the convergence of FPCP applied to a single panel, i.e., the consensus protocols in \eqref{e:proj} and \eqref{e:hinge_con}. Without loss of generality, we focus on the convergence of triangulated consensus in \eqref{e:proj}.}
{
\begin{theorem}[Theorem 5.1 in \cite{acc2025}]\label{the:local_c}
Consider the triangulated consensus dynamics in \eqref{e:proj}, where
$M_{\text{panel}}$ and $\tilde{\mathbf r}^{\text{panel}}_{i,j,k}$
are fixed during the local panel update. Then the projected
relative coordinates converge to the desired triangular formation, i.e.,
\[
M_{\text{panel}}\bigl(\mathbf{x}_{\iota}(t)-\mathbf{x}_{\kappa}(t)\bigr)
\rightarrow
\tilde{\mathbf r}^{\text{panel}}_{\iota}
-
\tilde{\mathbf r}^{\text{panel}}_{\kappa},
\qquad
\iota, \kappa\in\{i,j,k\}.
\]
Equivalently, the projected configuration converges to the desired
formation up to a translation in the local 2D frame.
\end{theorem}
}
{
The proof of Theorem \ref{the:local_c} is deferred to Appendix I.
}

{
Next, we will extend the convergence analysis to the weight-FPCP applied to a single panel.
\begin{corollary}[Convergence of the local W-FPCP]\label{t:W_conv}
Consider the local W-FPCP dynamics in \eqref{eqn:weighted_consensus}, 
where $\Omega_{i,j,k}$ in \eqref{e:weights} is symmetric positive
semidefinite and satisfies
\begin{equation}
    \ker(\Omega_{i,j,k})
    =
    \left\{
    (\mathbf{1}_3\otimes I_2)\mathbf c:
    \mathbf c\in\mathbb{R}^2
    \right\}.
    \label{eq:omega_nullspace}
\end{equation}
Then the projected relative coordinates converge to the desired
weighted formation, namely,
\[
M_{i,j,k}\left(\mathbf{x}_{\iota}(t)-\mathbf{x}_{\kappa}(t)\right)
\rightarrow
\tilde{\mathbf r}_{\iota}-\widetilde{\mathbf r}_{\kappa},
\qquad
\iota,\kappa\in\{i,j,k\}.
\]
Equivalently, the local W-FPCP converges to the desired projected
formation up to a rigid translation in the local 2D frame.
\end{corollary}
}

{
The proof of Corollary \ref{t:W_conv} is deferred to Appendix II.
}

{
Lastly, we examine the convergence of the aggregated W-FPCP in \eqref{e:homg} for the entire system, considering all panels and hinges.
\begin{assumption}\label{assump1}
Consider the aggregated W-FPCP dynamics in \eqref{e:homg} over a sequence of time intervals
$[t_q,t_{q+1})$, $q=0,1,\ldots$, with $t_q\to\infty$. The target formation $\tilde{\mathbf{r}}_{i, j, k}^{(\Gamma)}$, $\Gamma=1,\ldots,m+2h$, is fixed for the entire time. On each interval, the states are updated with fixed projection matrices. According to \eqref{e:weighted_alltogether2}, the dynamics for each time interval is rewritten as
\begin{equation}
    \dot{\mathbf{x}}
    = -\sum_{\Gamma=1}^{m+2h}
    \hat{W}_{\Gamma,q}(\mathbf{x}
    +{\mathbf b}_{\Gamma,q}),
    \; t\in[t_q,t_{q+1}),
    \label{eq:piecewise_agg_wfpcp}
\end{equation}
with subsystem $\mathcal{V}^{(\Gamma)} = \{i, j, k\}$, $\Gamma=1,\ldots,m+2h$.
\end{assumption}}

{
Assumption \ref{assump1} implies that for each time interval $q$, the system dynamics \eqref{eq:piecewise_agg_wfpcp} during that interval is a linear dynamical system, as the projection matrices hold constant. Thus, over the entire time, this assumption implies that the system has piecewise-linear dynamics.
Given that the target formation $\tilde{\mathbf{r}}_{i, j, k}^{(\Gamma)}$, $\Gamma=1,\ldots,m+2h$, is fixed for the entire time, we have the following common-equilibrium assumption.
\begin{assumption}\label{assump2}
For the dynamical system in \eqref{eq:piecewise_agg_wfpcp}, there exists a compatible equilibrium $\mathbf{x}^\ast$ such that
\begin{equation}
    H_q\mathbf{x}^\ast=\mathbf{g}_q,
    \qquad \forall q\ge 0,
    \label{eq:common_equilibrium}
\end{equation}
\end{assumption}
where $ H_q:=\sum_{\Gamma=1}^{m+2h}
    \hat{W}_{\Gamma,q}$, and $\mathbf{g}_q
    :=-\sum_{\Gamma=1}^{m+2h}
    \hat{W}_{\Gamma,q}{\mathbf b}_{\Gamma,q}$.
}

{
\begin{theorem}[Convergence of the aggregated W-FPCP]\label{t:a_conv}
Under Assumptions \ref{assump1} and \ref{assump2}, let $\mathcal N$ denote the fixed null-motion subspace associated with
global motion of the aggregated origami model. We assume that the fixed aggregated matrices $H_q$ at each time interval $[t_q,t_{q+1})$ have a common nullspace,
\[
    \ker(H_q)=\mathcal N,\qquad \forall q\ge 0.
\]
Assume that there exists $\mu>0$ such that
\begin{equation}
    \mathbf{e}^{\top}H_q\mathbf{e}
    \ge
    \mu \|\mathbf{e}\|^2,
    \qquad
    \forall \mathbf{e}\in\mathcal N^\perp,\ \forall q\ge 0.
    \label{eq:uniform_definiteness}
\end{equation}
Then the component of $\mathbf{x}(t)-\mathbf{x}^\ast$ orthogonal to
$\mathcal N$ converges to zero, i.e.,
$ P_{\mathcal N^\perp}\bigl(\mathbf{x}(t)-\mathbf{x}^\ast\bigr)
    \to 0$. Hence $\mathbf{x}(t)$ converges to the affine equilibrium set
$\mathbf{x}^\ast+\mathcal N$.
\end{theorem}
}
{
The proof of Theorem \ref{t:a_conv} is deferred to Appendix III.
}

\subsection{Parameter Determination for Weights of FPCP}\label{s:opt}
The edge weights embedded in each $6\times6$ weighted Laplacian $\Omega_{i,j,k}$ governs how fast and in which directions the vertices move under the FPCP protocol. Determining these weights is therefore important for
{improving the trajectory-reproduction fidelity of the proposed
consensus-based reconfiguration model. Consistent with the effective-parameter
interpretation above, these weights are identified from trajectory data.} With
captured vertex trajectories, the weight determination problem can be formulated
as an optimization problem. Specifically, the weights are chosen to minimize
differences between the {predictions of the consensus model}
and the observed vertex trajectories. This subsection provides the mathematical
formulation of the corresponding parameter-estimation problem and examines how
parameter optimization can improve the fidelity of the proposed
{consensus-based reconfiguration model.}

Assume that the coordinates of all vertices are observed at discrete times $t_{\kappa}$, $\kappa=1,\ldots,K$, with intervals $\Delta t_{\kappa} = t_{\kappa + 1} - t_{\kappa}$. Denote the experimentally measured vertex coordinates at each time step as  $\mathbf{y}(t_{\kappa})\in\mathbb{R}^{3\nu + 1}$. Correspondingly, the theoretical prediction from the consensus-based discrete-time model is denoted as $\mathbf{x}(t_{\kappa})\in\mathbb{R}^{3\nu + 1}$, whose evolution is governed by discretizing the continuous system dynamics in \eqref{e:homg}, expressed as
$$\mathbf{x}(t_{\kappa + 1}) = \left(I_{3\nu + 1} + A(\Omega) \Delta t_{\kappa} \right) \mathbf{x}(t_{\kappa})$$
where $A(\Omega)=\sum_{\Gamma=1}^{m+2h} A_{i, j, k}^{(\Gamma)}(\Omega)$, and each component matrix $A_{i, j, k}^{(\Gamma)}(\omega)$ depends on the corresponding project conversion matrix and weight parameters $\Omega^{(\Gamma)}$, $\Gamma=1,\ldots,m+2h$.

The parameter optimization aims to find the optimal set of weights $\Omega^{(\Gamma)}$ that minimize the cumulative mismatch between observed and predicted vertex positions. 
In summary, the parameter optimization problem is formulated as
%When optimizing to fit the actual system data, the system should be set as a discrete system. From this, we will need to get the matrix exponential of the system to get the discrete system. However, it is computationally expensive to calculate for the matrix exponential, so power series approximation will be used. 
\begin{subequations}\label{eq:design}
\begin{align}
    \min_{\Omega^{(1)}, \ldots, \Omega^{(m+2h)}}
   & \sum_{\kappa=1}^{K} || \mathbf{y}(t_{\kappa}) - \mathbf{x}(t_{\kappa}) ||_2 \label{eq:obj1}\\
    \text{s.t.}& \quad\mathbf{x}(t_{\kappa + 1}) = \left(I_{3\nu + 1} + A \Delta t_{\kappa} \right) \mathbf{x}(t_{\kappa})\nonumber\\
    & \quad\kappa=1,\ldots,K .
\end{align}
\end{subequations}
This constrained optimization problem can be formulated as nonlinear program (NLP) due to the bilinear dependence of the system states and the weights within the optimization. Existing NLP solvers, e.g., SNOPT, can efficiently address such problems and are employed to solve the optimization numerically.

%Performing this straightforward optimization to determine the weights of the panels serves as a foundation for analyzing more complex systems using these simplified equations. This approach streamlines calculations, enabling the system to efficiently identify a fitting model that can later be integrated into control mechanisms. Building on these fundamental optimizations allows more intricate models to be developed and constructed incrementally, providing a scalable framework for modeling and control applications.

\section{Simulation Results}
{ To demonstrate the effectiveness of W-FPCP in representing origami reconfiguration dynamics, we consider two representative cases: a simplified two-panel structure and a Kresling pattern. The two-panel structure is used as
a physics-based benchmark, where the reference trajectory is
generated from a bar-and-hinge mechanics model. This benchmark directly evaluates whether W-FPCP can reproduce vertex-coordinate trajectories generated by mechanical effects such as hinge stiffness, damping, and inertia. The Kresling pattern is then used to evaluate the proposed framework on a more complex origami reconfiguration process. Since a calibrated physics-based
bar-and-hinge with known material, hinge, damping, and inertia
parameters is not available for the Kresling structure considered here, we use time-discretized geometric reconfiguration data for parameter identification and trajectory comparison. Therefore, the Kresling example should be interpreted as a calibrated trajectory-reproduction study, while the two-panel case provides the mechanics-based validation benchmark. Since the W-FPCP model in \eqref{e:homg} is an autonomous dynamical system without external control input, both simulation cases present the folding/unfolding process from an intermediate unstable state to a stable equilibrium state without external control.

\subsection{Two-Panel Structure}
We first evaluate the proposed W-FPCP on a simplified two-panel origami structure with a mechanics-based reference trajectory. The structure consists of two triangular panels connected by a shared hinge. This example captures the essential hinge-driven unfolding behavior while allowing the reference
trajectory to be generated from a bar-and-hinge model that includes mechanical effects such as hinge stiffness, damping, and inertia derived from \cite{liu2017bar}.

In the bar-and-hinge model, the two panels rotate about the shared hinge from an initial fold angle of $100^\circ$ toward the equilibrium angle $180^\circ$, as shown Fig.~\ref{fig:traj_two_panel}. The mechanical parameters used in the bar-and-hinge simulation, including the panel geometry, bar stiffness, hinge stiffness, damping coefficient, and time step, are summarized in Table~\ref{tab:two_panel_parameters}. The resulting time histories of the panel vertices are used as the mechanics-based reference trajectory. 
\begin{table}[ht]
\centering
\begin{tabular}{|c|c|c|}
\hline
\textbf{Parameter} & \textbf{Number} & \textbf{Unit} \\
\hline
Axial bar stiffness & $1.0 \times 10^{4}$ & $\mathrm{N/m}$ \\
\hline
Rotational hinge stiffness & $10$ & $\mathrm{N\,m/rad}$ \\
\hline
Time step & $0.01$ & $\mathrm{s}$ \\
\hline
Translational damping coefficient & $6.0$ & $\mathrm{N\,s/m}$ \\
\hline
\end{tabular}
\caption{Physical parameters used in the bar-and-hinge dynamic model.}\label{tab:two_panel_parameters}
\end{table}

\vspace{-0.5cm}
\begin{figure}[!h]
    \centering
    \includegraphics[width=1\linewidth]{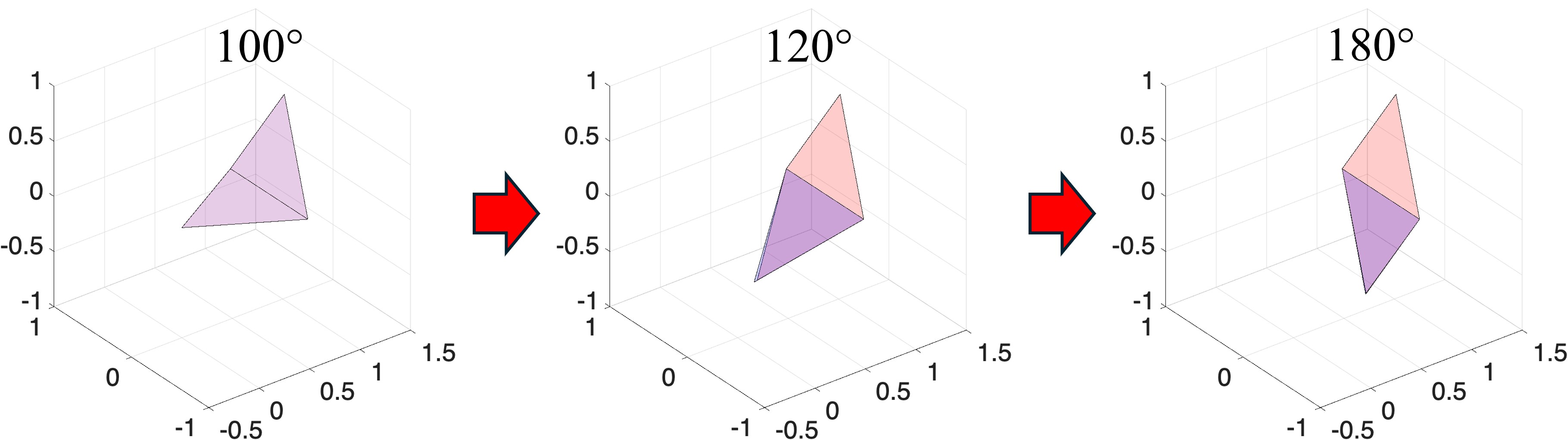}
    \caption{Reference bar-and-hinge model (red) and calculated (blue) states of the two-panel structure during the nominal unfolding process.}
    \label{fig:traj_two_panel}
\end{figure}

The W-FPCP model is constructed using two local triangular-panel subsystems. Each local triangular subsystem contains six projected weighting parameters.
These weights are first calibrated by minimizing the mismatch between the W-FPCP-predicted vertex trajectories and the bar-and-hinge reference trajectory for the nominal $100^\circ \to 180^\circ$ unfolding motion. As shown in
Fig.~\ref{fig:two_panel_nominal}, the calibrated W-FPCP closely reproduces the vertex trajectories generated by the bar-and-hinge model.

\begin{figure}[!h]
    \centering
    \includegraphics[width=0.9\linewidth]{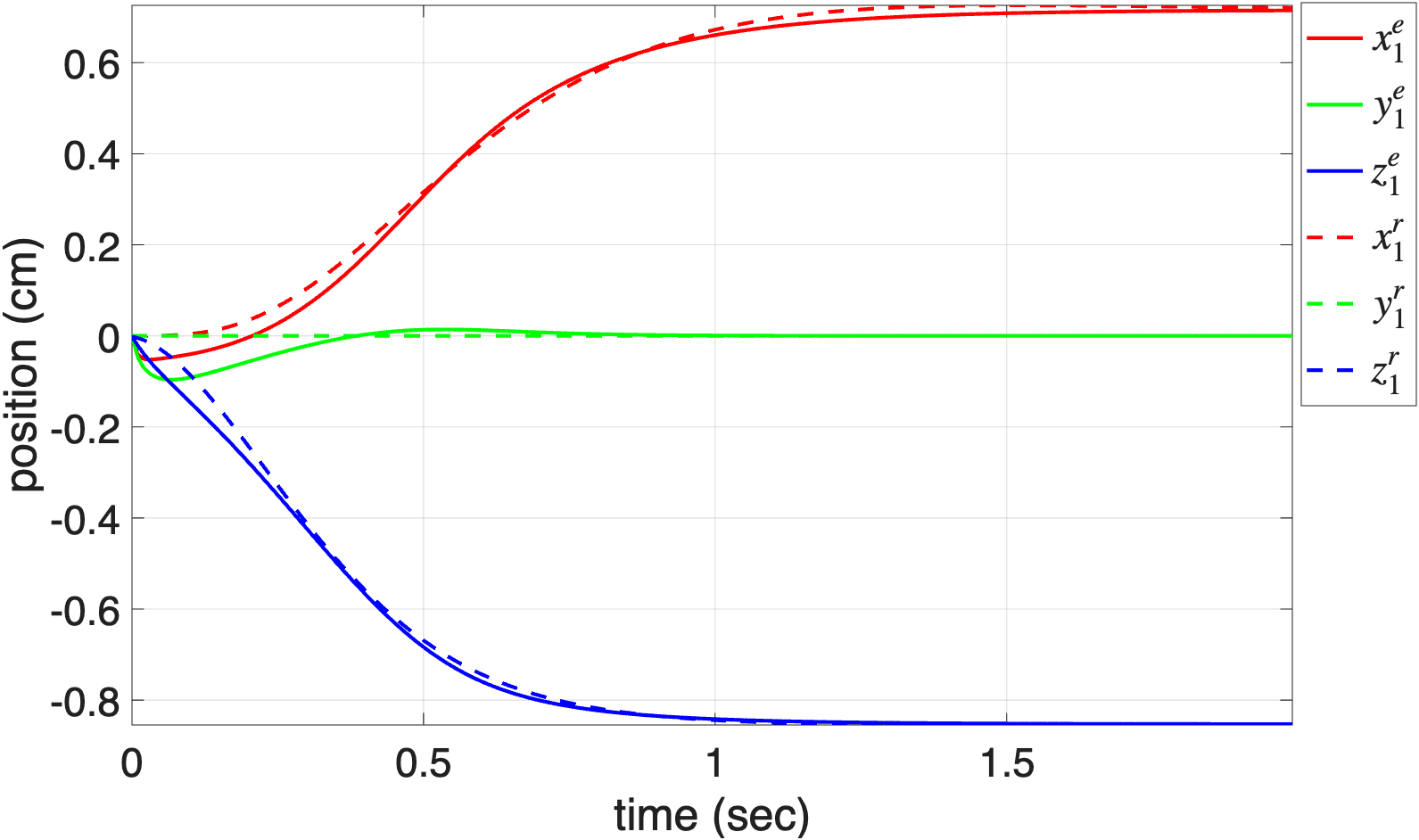}
    \caption{Comparison of coordinate history of trajectory two between calculated results (solid line) and reference data (dashed line) for Vertex 1 for two panel structure nominal simulation.}
    \label{fig:two_panel_nominal}
\end{figure}

To further examine the robustness of the calibrated weights, we keep the optimized weights fixed and simulate W-FPCP from perturbed initial vertex configurations. The structure rotates to a different orientation and starts with the same initial folding angle, while small vertex-level perturbations are introduced
to represent panel expansion and twisting disturbances. The corresponding bar-and-hinge trajectory is generated from the same perturbed initial configuration, and no re-optimization is performed. As shown in Figs.~\ref{fig:traj_two_panel} and \ref{fig:two_panel_noise}, the fixed-weight W-FPCP model remains able to reproduce the main trajectory trend under these initial perturbations.

\begin{figure}[!h]
    \centering
    \includegraphics[width=1\linewidth]{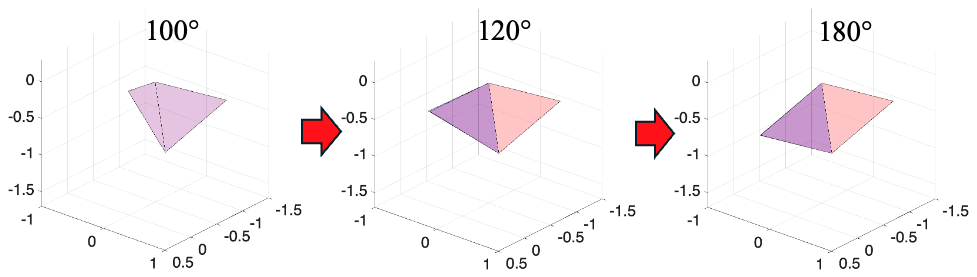}
    \caption{Reference bar-and-hinge model (red) and calculated (blue) states of the two-panel structure during the unfolding process with different configuration and noise.}
    \label{fig:traj_two_panel}
\end{figure}

\begin{figure}[!h]
    \centering
    \includegraphics[width=0.9\linewidth]{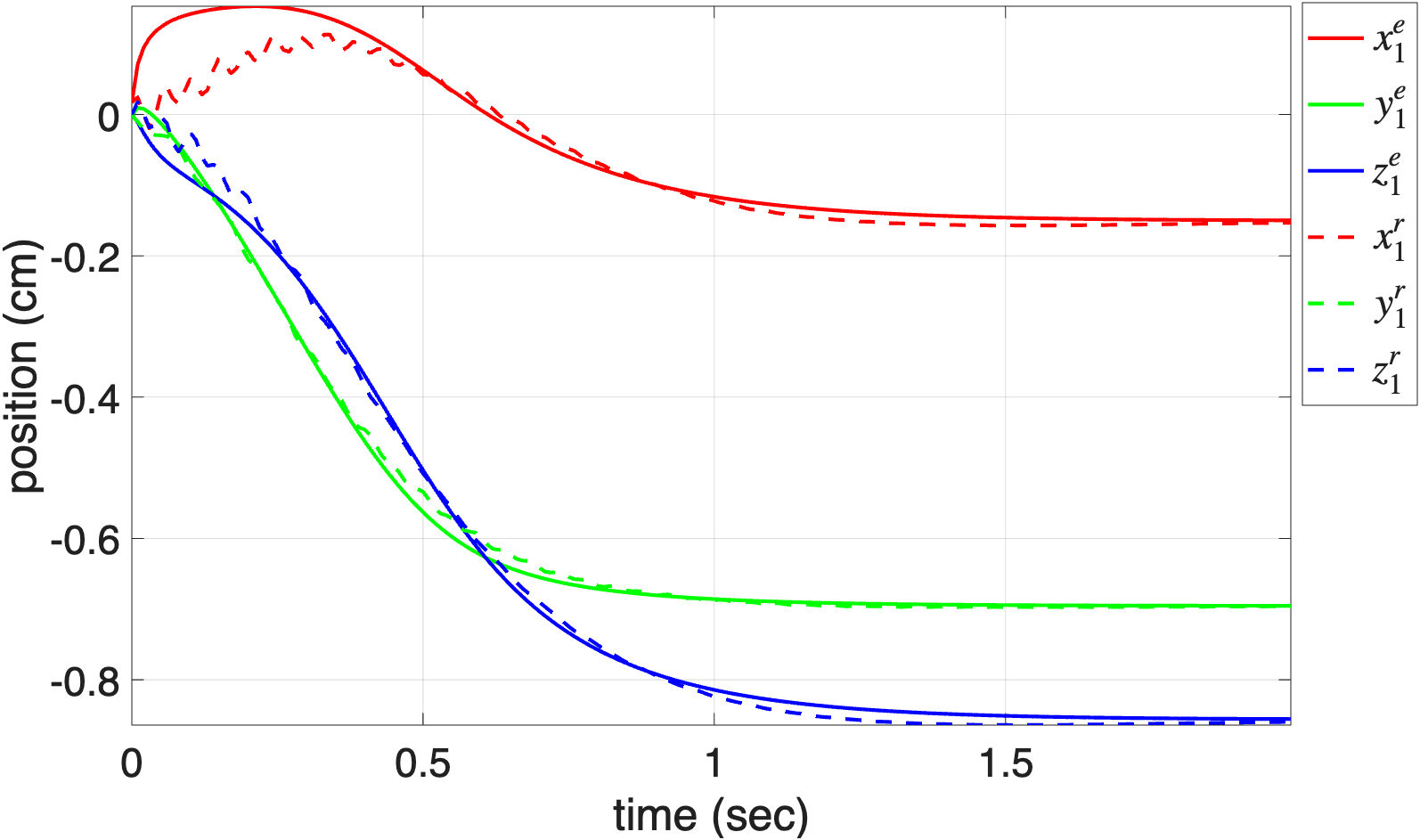}
    \caption{Comparison of coordinate history of trajectory two between calculated results (solid line) and reference data (dashed line) for Vertex 1 for two panel structure simulation with noise.}
    \label{fig:two_panel_noise}
\end{figure}
}

\subsection{Kresling Structure}  
The Kresling pattern is an origami structure that folds unidirectionally and exhibits bistability \cite{jianguo2015bistable}. When a continuous axial force, accompanied by a rotational torque, is applied along the central symmetric axis of the Kresling origami, it induces a unidirectional collapse or extension of the structure. All vertices move with uniform velocity along this axis at every moment, maintaining a synchronized folding motion.  

The Kresling pattern's bistable nature implies two stable equilibrium states, where the structure can maintain its configuration without additional external forces. The strain energy model governs the convergence to one of these equilibrium states, which depends on two key parameters at the initial unstable state: twist angle and height. We use the percentage of folding as a representative metric to simplify the characterization of the initial unstable states. If the initial folding percentage falls below a critical threshold, the structure stabilizes at state 1, corresponding to 0\% folded, as shown in Fig. \ref{f:pattern_param}({b}). In contrast, if the initial folding percentage exceeds this threshold, the structure settles at state 2, corresponding to 100\% folded, as shown in Fig. \ref{f:pattern_param}({c}). For a given Kresling origami design, this threshold can be determined on the basis of its geometric and material properties.  

\begin{figure}[!h]
\includegraphics[width=6.5cm]{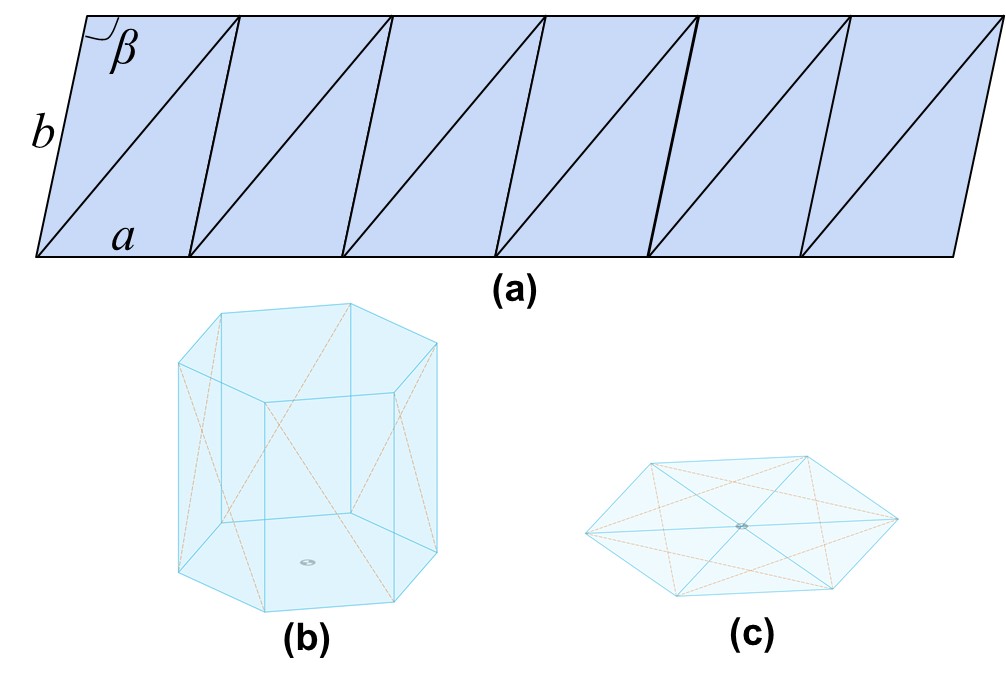}
\centering
\caption{An example of (a) Kresling pattern with parameters in their graph representations, (b) Equilibrium state 1 (0\% folded), and (c) Equilibrium state 2 (100\% folded).}\label{f:pattern_param}
\end{figure}

%\subsection{Panel and hinge Configurations of the Kresling Pattern}  
The Kresling pattern consists of uniform triangular panels that define the structural geometry. These panels are characterized by three fundamental parameters:  
The bottom length of the triangular panel is denoted as $ a$, the side length as $ b $, and the number of sides in the polygonal cross-section as $ n$. These parameters define the size and shape of each triangular panel within the structure, as shown in Fig. \ref{f:pattern_param}(a). %The stacked $\tilde{\mathbf{r}}_{i, j, k} = \begin{bmatrix} \tilde{\mathbf{r}}^{\top}_{i} & \tilde{\mathbf{r}}^{\top}_{j} & \tilde{\mathbf{r}}^{\top}_{k}\end{bmatrix}$.
Based on these geometric properties, the desired triangular configuration, $\tilde{\mathbf{r}}^{\text{panel}}_{i, j, k}=[\tilde{\mathbf{r}}^{\text{panel}\top}_i,\tilde{\mathbf{r}}^{\text{panel}\top}_j,\tilde{\mathbf{r}}^{\text{panel}\top}_k]^{\top}$, of the Kresling pattern can be formulated as  
\begin{subequations}
\begin{align}
    \tilde{\mathbf{r}}^{\text{panel}}_{i} =& \begin{bmatrix} 0 & 0 \end{bmatrix}^{\top}, \\  
    \tilde{\mathbf{r}}^{\text{panel}}_{j} =& \begin{bmatrix} a & 0 \end{bmatrix}^{\top}, \\  
    \tilde{\mathbf{r}}^{\text{panel}}_{k} =& \begin{bmatrix} -b \cos\beta & b \sin\beta \end{bmatrix}^{\top}, 
    \end{align}
\end{subequations}
where $ \beta$ is defined as
\begin{equation}
\beta = \frac{\pi}{n} + \arcsin\left(\frac{b}{a} \sin\left(\frac{\pi}{n}\right)\right).
\end{equation}

%\subsection{Hinge Configuration of the Kresling Pattern}
The hinge configuration in the Kresling pattern defines the relative angles between two adjacent panels during transitions between the folded and unfolded states, that allows panel deformations while maintaining the structural integrity of the origami pattern.
In the simulation, the system target state is set in a fully {folded and }unfolded configuration, i.e., the equilibrium state 1 in Fig. \ref{f:pattern_param}(b), where the target configurations for hinge consensus are determined.
There are two distinct hinge configurations in the Kresling pattern, each associated with the corresponding target formation, $\tilde{\mathbf{r}}_{i,k,l}^{\text{hinge1}}$ and \(\tilde{\mathbf{r}}_{i, j, l}^{\text{hinge2}}\), defined in \eqref{e:hinge_desired}. %. The three vertices of the hinge configuration are determined based on fully unfolded states to ensure that the structural parameters align with the geometric constraints. The desired configurations, $\tilde{\mathbf{r}}_{i,k,l}^{\text{hinge1}}$ and \(\tilde{\mathbf{r}}_{i, j, l}^{\text{hinge2}}\) defined in \eqref{e:hinge_desired}, are determined by the following components
For the configuration with $\phi=0$, $\mathbf{q}_{\iota, \kappa}$, $\iota,\kappa=i,j,k,{l},\,\iota\neq\kappa$ are defined by
\begin{subequations}
\begin{align}
    \mathbf{q}_{k, l} =& \mathbf{q}_{i, j} = \begin{bmatrix} -b \cos\beta & -b \sin\beta & 0 \end{bmatrix} ^\top, \\
       \mathbf{q}_{l, j} =& \begin{bmatrix} -a & 0 & 0\end{bmatrix} ^\top, \\
    \mathbf{q}_{k, j} =& \begin{bmatrix} -a-b \cos\beta & -b \sin\beta & 0 \end{bmatrix} ^\top.
\end{align}
\end{subequations}
Similarly, for the configuration with $\phi=\pi - \frac{(n-2)\pi}{n}$,  $\mathbf{q}_{\iota, \kappa}$, $\iota,\kappa = i,j,k,{l,}\,\iota\neq\kappa$ are defined by
%From these flat-state configurations, the hinge parameters \(\tilde{\mathbf{r}}_{i, k, l}^{\text{hinge1}}\) and \(\tilde{\mathbf{r}}_{i, j, l}^{\text{hinge2}}\) can be determined as described in \eqref{e:hinge_desired}. These calculated hinge states enable the Kresling pattern to undergo smooth transformations while adhering to the consensus-based dynamic model.
% Similarly, the second hinge type follows a corresponding set of configurations:
\begin{subequations}
\begin{align}
    \mathbf{q}_{k, l} =& \mathbf{q}_{i, j}=\begin{bmatrix} a & 0 & 0 \end{bmatrix} ^\top, \\
        \mathbf{q}_{l, j} =& \begin{bmatrix} - a - b\cos{\beta} & -b\sin{\beta} & 0\end{bmatrix} ^\top, \\
    \mathbf{q}_{k, j} =& \begin{bmatrix} -b \cos\beta & -b \sin\beta & 0 \end{bmatrix} ^\top.
\end{align}
\end{subequations}
The geometric dimensions of the Kresling pattern are set as $a=5$, $b=8.66$, and $n=6$.
% \begin{subequations}
% \begin{align}
%     \tilde{\mathbf{s}}_{i}^{hinge2} =& \begin{bmatrix} -b \cos\beta - a & c \sin\beta \end{bmatrix} ^\top, \\
%     \tilde{\mathbf{s}}_{j}^{hinge2} =& \begin{bmatrix} 0 & 0 \end{bmatrix} ^\top, \\
%     \tilde{\mathbf{s}}_{k}^{hinge2} =& \begin{bmatrix} -b \cos\beta + a & c \sin\beta \end{bmatrix} ^\top, \\
%     \tilde{\mathbf{s}}_{l}^{hinge2} =& \begin{bmatrix} a & 0 \end{bmatrix} ^\top.
% \end{align}
% \end{subequations}

%These hinge configurations serve as the fundamental elements in the Kresling pattern, allowing for controlled folding and unfolding motions. By employing Frame Projection Consensus control, the hinges facilitate a smooth and synchronized structural transformation, ensuring that the origami pattern adheres to its predefined dynamic constraints.

%This formulation provides a precise mathematical representation of the triangular panels in the Kresling pattern, facilitating its integration into dynamic network models for controlled folding and unfolding behavior.

%\subsection{Model and Parameter Selection}
To incorporate weighting parameters into the proposed FPCP, the parameter optimization problem formulated in \eqref{eq:design} is numerically solved to obtain optimal weights associated with each Laplacian matrix. For demonstration purposes, the observed reference data utilized in this simulation is generated from a geometric kinematic model, assuming the origami structure folds at a constant angular rotation speed \cite{lu2022conical}. To effectively capture and reflect diverse dynamic behaviors and enhance robustness, randomized time intervals, $\Delta t_{\kappa}$, are selected throughout the simulation. {In this simulation example, time variable $t$ here denotes a continuous-time evolution parameter associated with the consensus dynamics. Although it is not explicitly identified with physical actuation time, it serves as a time-like variable that parameterizes the progression of the system during the reconfiguration process.} 
%This variation ensures that the optimization framework can accommodate different material properties, leading to a robust and adaptable design.
%fulfill the required material properties, an optimization framework is applied, ensuring that the system adheres to the structural constraints and dynamic behavior introduced in Section \ref{s:weighted}. This approach refines the parameters to achieve optimal performance while maintaining the integrity of the origami structure. 

The optimized weights are integrated into the dynamic model, and the resulting simulated vertex trajectories are compared directly against the {two} reference trajectories from the geometric model {with one set of weights}. The simulated vertex coordinate histories using the optimized parameters and their corresponding reference trajectories are illustrated in {Figs.~\ref{fig:vertex_one_result1} and \ref{fig:vertex_one_result2}. Trajectory 1 in Fig. \ref{fig:vertex_one_result1} starts from an initially unstable, 25\%-folded configuration toward the equilibrium (0\%-folded), while trajectory 2 in Fig. \ref{fig:vertex_one_result2} starts from 75\%-folded configuration toward another equilibrium (100\%-folded) state.} As clearly depicted in these plots, the vertex positions predicted by the optimized dynamic model exhibit a close agreement with the reference data. {The trajectories generated using optimized weights closely reproduce the reference coordinate histories and terminal configurations for the selected geometric reconfiguration cases.}
% The dynamics generated using optimized weights accurately replicate both transient behaviors and equilibrium states observed in the reference trajectories.

%Compared to the observed data from a geometric model, the time histories of coordinates obtained from the dynamics with optimized parameters closely match the coordinate histories from the observed data. The mean squared error between the calculated results and observed data is summarized in Table ????   %demonstrating the system's response under the defined criteria.
\begin{figure}[!h]
    \centering
    \includegraphics[width=1\linewidth]{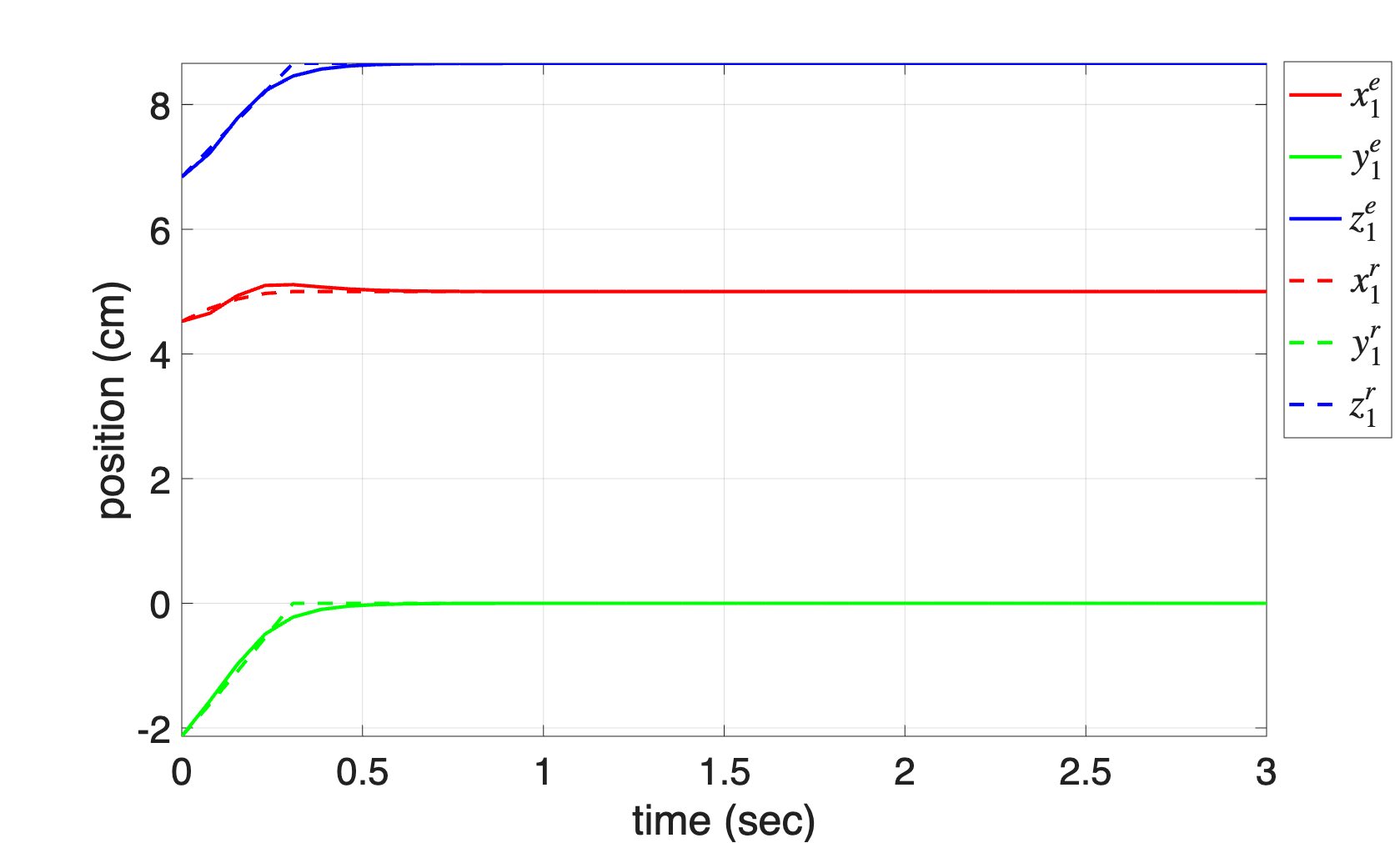}
    \caption{Comparison of coordinate history {of trajectory 1} between calculated results {(solid line)} and reference data {(dashed line)} for Vertex 1.}
    \label{fig:vertex_one_result1}
\end{figure}

\begin{figure}[t]
    \centering
    \includegraphics[width=\columnwidth]{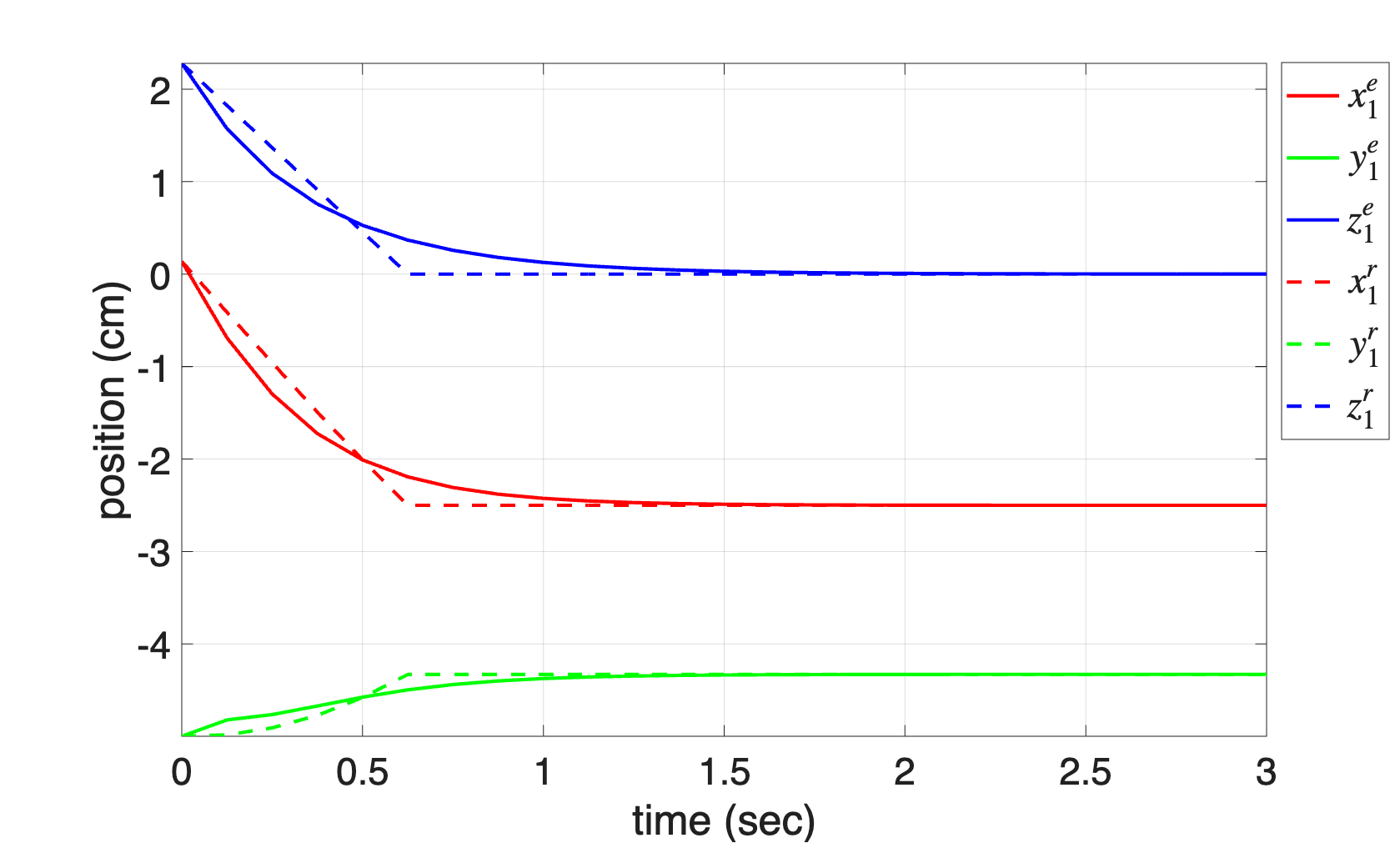}
    \caption{{Comparison of coordinate history of trajectory 2 between calculated results (solid line) and reference data (dashed line) for Vertex 1.}}
    \label{fig:vertex_one_result2}
\end{figure}

{To} quantitatively assess the accuracy of the parameter identification procedure, the mean squared errors (MSE) between the simulated vertex positions (obtained with optimized parameters) and the reference trajectories (obtained from the geometric model) are computed. These error metrics are summarized in Table~\ref{tab:optimization_mse}. 

\begin{table}[ht]
  \centering
  \begin{tabular}{|l|c|c|}
    \hline
     & Non-weighted FPCP & W-FPCP with optimized weights\\ \hline
    MSE & \iffalse 181.5791 \fi {129.4218} & \iffalse 0.0649 \fi  {0.3276}\\ \hline
  \end{tabular}
    \caption{MSE of non-weighted FPCP and W-FPCP relative to reference trajectory}\label{tab:optimization_mse}
\end{table}

Using the optimized parameters determined by the W-FPCP, Fig. \ref{fig:optimization_results} illustrates the dynamic unfolding/folding process of the Kresling origami pattern, respectively. Simulation results demonstrate that, under the W-FPCP with optimized weights, the origami structure naturally evolves {toward} the stable state 1 (fully unfolded) {or stable state 2 (fully folded)}. The simulated states closely match the reference states derived from the geometric kinematic model, {indicating that the proposed model accurately reproduces the observed reconfiguration trajectory}. The congruence between simulation results from the W-FPCP and the reference data {supports the effectiveness of} the proposed weighted consensus modeling framework {for representing origami reconfiguration behavior.}

{ The single optimized weight set in this example is interpreted as a compact effective parameterization for the selected Kresling reconfiguration trajectories, rather than a globally valid parameter set for all folding percentages.
Because the Kresling dynamics are nonlinear and branch-dependent, fitting the weights over the entire range of folding percentages can reduce trajectory-specific accuracy, while fitting only a short segment may degrade prediction outside that segment. This reflects an expressiveness limitation of using one constant set of W-FPCP weights. Developing piecewise multi-set weighting strategies for different folding regimes is left for future work.}
\vspace{-0.3cm}
\begin{figure}[!h]
    \centering
    \includegraphics[width=0.9\linewidth]{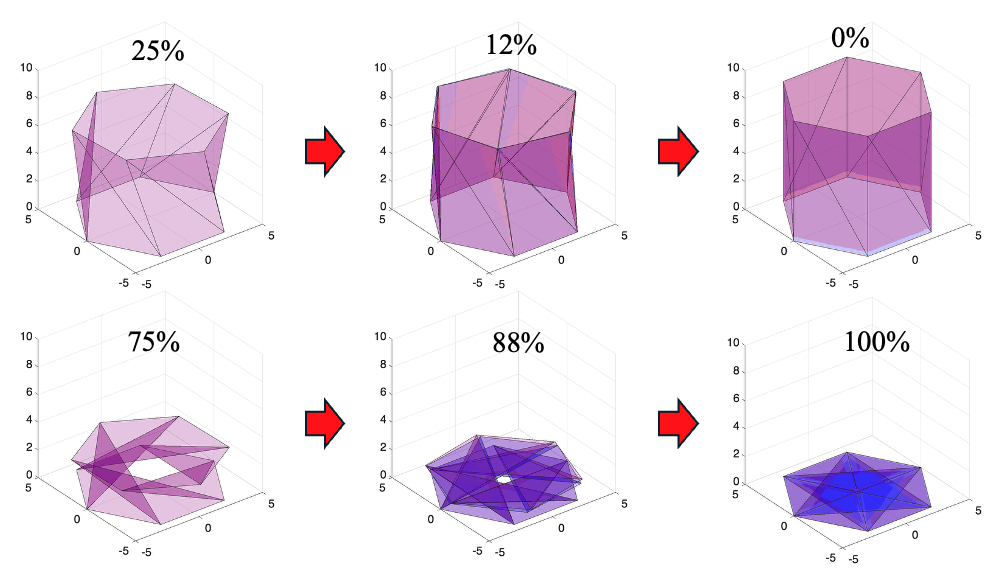}
    \caption{Reference (red) and calculated (blue) states of the Kresling pattern during the unfolding process {of trajectories 1 and 2}.}
    \label{fig:optimization_results}
\end{figure}

\section{Concluding Remarks}
This paper proposed the frame-projected consensus protocol (FPCP) to model the folding and unfolding dynamics of triangulated origami structures. By projecting each local vertex state onto a local two-dimensional frame of its panel or hinge plane before applying the consensus protocol, FPCP preserves the geometry of every panel throughout reconfiguration. The weighted extension of FPCP, named W-FPCP, is then proposed to better embed the physical attributes of the origami structure and {improve} the modeling fidelity. Theoretical analysis of the proposed method {to establish} its stability and convergence properties has also been provided. In addition, a data-fitting parameter estimation framework has been introduced to estimate model weights from observed vertex trajectories. {Numerical simulations on a two-panel structure and a Kresling pattern demonstrated that the proposed W-FPCP closely tracks the reference data and naturally converges to an equilibrium condition, even under initial perturbations. These results support the ability of the proposed framework to represent origami reconfiguration trajectories in agreement with the chosen geometric reference model.} Future work will extend the proposed framework to enable online parameter adaptation, {multi-set of parameter fitting}, and validation of the model experimentally on various origami prototypes.
\vspace{-0.4cm}
{
\section{Appendix I\\ Proof of Theorem 4.2}
\begin{proof}
Define $
\boldsymbol{\delta}
=
\tilde{M}_{\rm panel}\mathbf{x}_{i,j,k}
-
\tilde{\mathbf r}^{\text{panel}}_{i,j,k}$. Since
$
\tilde{M}_{\text{panel}}\tilde{M}_{\text{panel}}^{\top}
=
I_3\otimes(M_{\text{panel}}M_{\text{panel}}^{\top})
=
I_6,
$ the dynamics in \eqref{e:proj} implies
\[
\dot{\boldsymbol{\delta}}
=
\tilde{M}_{\text{panel}}\dot{\mathbf{x}}_{i,j,k}
=
-L_2\boldsymbol{\delta}.
\]
Thus, the projected error follows the standard formation-consensus
dynamics in the local 2D frame. Let
$
V(\boldsymbol{\delta})=\frac{1}{2}\|\boldsymbol{\delta}\|^2, 
$
then
\[
\dot V
=
\boldsymbol{\delta}^{\top}\dot{\boldsymbol{\delta}}
=
-\boldsymbol{\delta}^{\top}L_2\boldsymbol{\delta}
\le 0,
\]
where $L_2=L(K_3)\otimes I_2$ is symmetric positive semidefinite.
By LaSalle's invariance principle, $\boldsymbol{\delta}(t)$ converges to
the largest invariant set satisfying $L_2\boldsymbol{\delta}=0$. Since
$K_3$ is connected,
\[
\ker L_2
=
\left\{
(\mathbf{1}_3\otimes I_2)\mathbf c:
\mathbf c\in\mathbb{R}^2
\right\}.
\]
Therefore, $
\boldsymbol{\delta}(t)
\rightarrow
(\mathbf{1}_3\otimes I_2)\mathbf c
$ for some $\mathbf c\in\mathbb{R}^2$, or equivalently,
\[
\tilde{M}_{\text{panel}}\mathbf{x}_{i,j,k}(t)
\rightarrow
\tilde{\mathbf r}^{\text{panel}}_{i,j,k}
+
(\mathbf{1}_3\otimes I_2)\mathbf c.
\]
Taking the differences between any two vertices eliminates the translation
term and gives
\[
M_{\text{panel}}\bigl(\mathbf{x}_{\iota}(t)-\mathbf{x}_{\kappa}(t)\bigr)
\rightarrow
\tilde{\mathbf r}^{\text{panel}}_{\iota}
-
\tilde{\mathbf r}^{\text{panel}}_{\kappa},
\qquad
\iota, \kappa\in\{i,j,k\}.
\]
This proves convergence to the desired projected triangular formation,
up to a rigid translation in the local 2D frame.
\end{proof}}
\vspace{-0.4cm}
{
\section{Appendix II\\ Proof of Corollary 2.1}
\begin{proof}
The proof follows directly from Theorem \ref{the:local_c} by replacing the unweighted
Laplacian $L_2$ with the weighted Laplacian $\Omega_{i,j,k}$. Similarly, for
$
\boldsymbol{\delta}
=
\tilde{M}_{i,j,k}\mathbf{x}_{i,j,k}
-
\tilde{\mathbf r}_{i,j,k},
$ we have
$
\dot{\boldsymbol{\delta}}
=
-\Omega_{i,j,k}\boldsymbol{\delta},
$
where
$\tilde{M}_{i,j,k}\tilde{M}_{i,j,k}^{\top}=I_6$.
Since $\Omega_{i,j,k}$ is symmetric positive semidefinite, the 
W-PFCP for a single panel gives
\[
\boldsymbol{\delta}(t)
\rightarrow
\ker(\Omega_{i,j,k}).
\]
Using the null-space condition in \eqref{eq:omega_nullspace}, there
exists $\mathbf c\in\mathbb{R}^2$ such that
$
\boldsymbol{\delta}(t)
\rightarrow
(\mathbf{1}_3\otimes I_2)\mathbf c$.
Similar to Theorem \ref{the:local_c}, we have
\[
M_{i,j,k}\left(\mathbf{x}_{\iota}(t)-\mathbf{x}_{\kappa}(t)\right)
\rightarrow
\widetilde{\mathbf r}_{\iota}-\widetilde{\mathbf r}_{\kappa},
\qquad
\iota, \kappa\in\{i,j,k\}.
\]
This completes the proof.
\end{proof}}
\vspace{-0.4cm}
{
\section{Appendix III\\ Proof of Theorem 4.5}
\begin{proof}
Using the definitions of $H_q$ and $\mathbf{g}_q$, the dynamics on
$[t_q,t_{q+1})$ can be written as
\[
    \dot{\mathbf{x}}
    =
    -H_q\mathbf{x}+\mathbf{g}_q.
\]
By the common-equilibrium condition \eqref{eq:common_equilibrium},
$H_q\mathbf{x}^\ast=\mathbf{g}_q$ for every $q$. Defining
$\mathbf{e}(t)=\mathbf{x}(t)-\mathbf{x}^\ast$,
we obtain the piecewise-linear error dynamics
\[
    \dot{\mathbf{e}}=-H_q\mathbf{e},
    \qquad t\in[t_q,t_{q+1}).
\]
From Corollary \ref{t:WequalOmega}, $\hat{W}_{\Gamma,q}$, $\Gamma=1,\ldots,m+2h$, is positive semidefinite. Hence, each
$H_q$ is positive semidefinite. Moreover, by definition,
$H_q\mathbf{n}=0$ for every $\mathbf{n}\in\mathcal N$ and every $q$.
Since $\mathcal N^\perp$ is invariant under $H_q$, the error component
$\mathbf{e}_\perp:=P_{\mathcal N^\perp}\mathbf{e}
$ satisfies
$\dot{\mathbf{e}}_\perp=-H_q\mathbf{e}_\perp$. 
Consider
$
    V(\mathbf{e}_\perp)=\frac{1}{2}\|\mathbf{e}_\perp\|^2.
$ For $t\in[t_q,t_{q+1})$,
\[
    \dot V
    =
    \mathbf{e}_\perp^\top \dot{\mathbf{e}}_\perp
    =
    -\mathbf{e}_\perp^\top H_q\mathbf{e}_\perp.
\]
Using \eqref{eq:uniform_definiteness}, we obtain
\[
    \dot V
    \le
    -\mu\|\mathbf{e}_\perp\|^2
    =
    -2\mu V.
\]
Therefore, $
    V(t)\le e^{-2\mu t}V(0)$.
Hence, on each interval,
\[
    V(t)
    \le
    e^{-2\mu(t-t_q)}V(t_q),
    \qquad t\in[t_q,t_{q+1}).
\]
Since $\mathbf e_\perp(t)$ is continuous at
each switching time $t_q$, the value of the common Lyapunov function does not jump, i.e., $
    V(t_q^+)=V(t_q^-)$.
Putting the interval inequalities in sequences gives
\[
    V(t)\le e^{-2\mu(t-t_0)}V(t_0),
    \qquad t\ge t_0.
\]
Consequently,
\[
    \|\mathbf e_\perp(t)\|
    \le
    e^{-\mu(t-t_0)}\|\mathbf e_\perp(t_0)\|,
\]
which proves exponential convergence of the component orthogonal to
$\mathcal N$.
\end{proof}}

\bibliographystyle{IEEEtran}
\bibliography{main_bib}
\vspace{-1.3cm}
\end{document}